\documentclass[12pt]{amsart}
\input{xy}
\xyoption{all}
\usepackage{amssymb}

\newtheorem{theorem}{Theorem}[section]
\newtheorem{lemma}[theorem]{Lemma}

\newtheorem{proposition}[theorem]{Proposition}
\newtheorem{corollary}[theorem]{Corollary}

\numberwithin{equation}{section}

\def\umono{\ar@{_{(}->}[u]}
\def\uumono{\ar@{_{(}->}[uu]}

\def\lmono{\ar@{_{(}->}[l]}
\def\llmono{\ar@{_{(}->}[ll]}

\newcommand{\Hom}{{\rm Hom \,}}

\newcommand{\Z}{{\mathbb Z}}

\def\nor{\trianglelefteq}
\def\sdp{\rtimes}
\def\nle{\not\le}
\def\iso{\cong}
\def\ord#1{\vert #1 \vert}
\def\ind#1#2{\vert #1\, :\, #2\vert}
\def\gp#1{\langle \, #1 \, \rangle}
\def\phi{\varphi}
\def\Phi{\varPhi}

\def\ZZ{\Bbb Z}

\def\NN{\Bbb N}
\def\FF{\Bbb F}

\def\L{\mathcal L}
\def\O{\mathcal O}
\def\Aut{\text{Aut}}

\def\dim{\text{dim }}
\def\isom{\text{Isom}}
\def\divides{\bigm |}
\def\notdivides{\not\kern 2.2pt\bigm |}
\def\tensor{\otimes}
\def\Res{\text{Res}}
\def\Ind{\text{Ind}}
\def\Coind{\text{Coind}}
\def\Hom{\text{Hom}}
\def\omegabar#1{\frak A_1(#1)}

\title[Strongly closed subgroups]
{Strongly closed subgroups of finite groups}

\author{Ram\'on J. Flores and Richard M. Foote}
\thanks{The first author was supported by MEC grant MTM2004-06686.}

\date{\today}
\begin{document}

\begin{abstract}
This paper gives a complete classification of the finite
groups that contain a strongly closed $p$-subgroup for $p$ any prime.

\end{abstract}

\maketitle

\section{Introduction}

For any finite group $G$ and subgroup $S$ we say two elements of $S$
are {\it fused} in $G$ if they are conjugate in $G$
but not necessarily in $S$.
This concept has played a central
role in group theory and representation theory, particularly in the case
when $S$ is a Sylow $p$-subgroup of $G$ for $p$ a prime.
A subgroup $A$ of $S$ is called {\it strongly closed}
in $S$ with respect to $G$ if for every $a \in A$, every element
of $S$ that is fused in $G$ to $a$ lies in $A$; in other words,
$a^G \cap S \subseteq A$, where $a^G$ denotes the $G$-conjugacy class
of $a$.  It is easy to verify that if $A$ is a $p$-subgroup, then $A$
is strongly closed in a Sylow $p$-subgroup if and only if it is strongly
closed in $N_G(A)$, so the notion of strong closure for a $p$-subgroup
does not depend on the Sylow subgroup containing it.
For a $p$-group $A$ we therefore simply say $A$ is strongly closed.
Seminal works in the theory of strongly closed 2-subgroups are
the celebrated Glauberman $Z^*$-Theorem, \cite{Gla66}, and
Goldschmidt's theorem on strongly closed abelian 2-subgroups,
\cite{Gol74}.
The $Z^*$-Theorem proved that if $A$ is strongly closed and of order 2,
then $\overline A \le Z(\overline G)$, where the overbars denote passage to
$G/O_{2'}(G)$.
Goldschmidt extended this by showing that if $A$ is a strongly closed
abelian 2-subgroup, then $\gp{ \overline A^{\overline G}}$
is a central product of an abelian 2-group and quasisimple groups that
either have a $BN$-pair of rank 1 or have abelian Sylow 2-subgroups.
These two theorems, in particular, played fundamental roles in the
study of finite groups, especially in the Classification of the
Finite Simple Groups.
The purpose of this paper is to give a classification of all finite groups
containing a strongly closed $p$-subgroup for an arbitrary prime $p$
(not assuming the strongly closed subgroup is abelian).

The concept of strong closure has important ramifications beyond finite group theory.
In particular, it is intimately connected to Puig's formulation
of {\it fusion systems} (or Frobenius categories), which evolved from the
modular representation theory of finite groups: To each $p$-block of a
finite group one can associate a (saturated) fusion system.
Puig's axiomatic approach provided the formalism necessary to study fusion
in a context which subsumes, as a special case, the natural fusion system arising
from pairs $(G,S)$, where $G$ is a finite group and $S$ is a Sylow $p$-subgroup
of $G$.
The concept of strong closure extends in an obvious way to abstract fusion systems
and plays a critical role therein: If $\mathcal F$ is a fusion system on a $p$-group $S$,
then the homomorphic images of $\mathcal F$ are in bijective correspondence
with the strongly closed subgroups of $S$.
Fusion systems were further refined by Broto, Levi, and Oliver in
\cite{Broto032} to create the class of $p$-local finite groups
(see also \cite{AS07}, \cite{Broto207}, \cite{Broto07} and \cite{Linckelmann06}).
Oliver then used this approach to prove that the homotopy type
of the $p$-completed classifying space of a finite group $G$ is uniquely
determined by the saturated fusion system $(G,S)$,
where $S$ is a Sylow $p$-subgroup of $G$.
Thus strong closure and its extensions to fusion systems and $p$-local finite
group theory also has significant ramifications in deep and currently very active
areas of modular representation theory and algebraic topology.

This paper is also the group-theoretic result needed for a classification
theorem in homotopy theory, which was the original impetus for our joint work.
Groups containing a strongly closed 2-subgroup
were characterized earlier in \cite{Foote97}, and that theorem formed the
underpinning of a complete description of the
$B\Z/2$-cellularization (in the sense of Dror-Farjoun) of the
classifying spaces of all finite groups, \cite{Ramon} and \cite{Flores07}.
In order to correspondingly describe the
$B\Z/p$-cellular\-ization of classifying spaces for odd primes $p$, we needed the
classification of finite groups containing a strongly closed $p$-subgroup
for odd $p$ --- this is the main theorem herein.
The complete description of the cellular structure
(with respect to $B\Z/p$)
of classifying spaces for
all finite groups and all primes $p$ is then established in
the separate paper \cite{Flores-Foote08}.
Our two classifications, the latter relying on the former, epitomize
the rich interplay between their subject areas that has
historically been evident and is currently even more vibrant.

A curious application of strong closure to ordinary representation theory and
number theory appears in \cite{Foote97a}.

Finally,
although the techniques used in this paper are purely group-theoretic,
the underlying fusion arguments provide deeper insight into
topological considerations in our second classification.  Indeed,
the marriage of these elements is seen in high relief in
Section~\ref{examples} where we explore more explicit
configurations that give rise in \cite{Flores-Foote08}
to interesting --- what might be called exotic --- classifying spaces.

\medskip
\noindent
{\bf Acknowledgements}.
We thank Carles Broto, David Dummit, Bob Oliver
and J\'er\^ome Scherer for helpful discussions, and also
George Glauberman for providing a motivating example.
We thank Michael Aschbacher for sharing his lecture notes and broad perspective,
and we acknowledge that a number of the results in Section~3 were
also proven independently by him.

\subsection{Statement of Results}
\label{statement-of-results}
%
\    

To describe the main results we introduce some new notation.
Henceforth $p$ is any prime, $S$ is a Sylow $p$-subgroup of the
finite group $G$ and $A$ is a subgroup of $S$.
In general let $R$ be any
$p$-subgroup of $G$. If $N_1$ and $N_2$ are normal subgroups of
$G$ with $R \cap N_i \in Syl_p(N_i)$ for both $i=1,2$, then $R
\cap N_1 N_2$ is a Sylow $p$-subgroup of $N_1 N_2$.  Thus there is
a unique largest normal subgroup $N$ of $G$ for which $R \cap N
\in Syl_p(N)$; denote this subgroup by $\O_R (G)$. Thus
$$
R \text{ is a Sylow $p$-subgroup of $\gp{R^G}$ if and only if } R
\le \O_R (G).
$$
Note that $O_{p'}(G/\O_R (G)) = 1$; in particular, if $R = 1$ is
the identity subgroup then $\O_1(G) = O_{p'}(G)$.
In general, $R \O_R(G) / \O_R(G)$ does not contain the
Sylow $p$-subgroup of any nontrivial normal subgroup of $G/\O_R(G)$;
in other words, $\O_{\overline R}(\overline G) = 1$,
where overbars denote passage to $G/\O_R(G)$.
Throughout the paper we freely use the
observation that strong closure passes to quotient groups
(cf.~Lemma~\ref{basic-facts}), so when analyzing groups where
$R \nle \O_R(G)$ we may factor out $\O_R(G)$.
With this in mind, the classification for strongly closed 2-subgroups
from \cite{Foote97} is as follows:

\begin{theorem}
\label{theorem3-1-p-equal-2}
Let $G$ be a finite group that
possesses a strongly closed 2-subgroup $A$. Assume $A$ is not a
Sylow 2-subgroup of $\gp{A^G}$, and let
$\overline G = G /\O_A(G)$.
Then $\overline A \ne 1$ and
$\gp{\overline A^{\overline G}} = L_1 \times L_2 \times \cdots \times L_r$,
where each $L_i$ is isomorphic to $U_3(2^{n_i})$ or $Sz(2^{n_i})$ for some $n_i$, and
$\overline A \cap L_i$ is the center of a Sylow 2-subgroup of $L_i$.
\end{theorem}

The classification for $p$ odd, which is the principal objective
of the paper, yields a more diverse set of ``obstructions'' with
added ``decorations'' as well.

\begin{theorem}
\label{theorem3-1}
Let $p$ be an odd prime and let $G$ be a finite group that
possesses a strongly closed $p$-subgroup $A$. Assume $A$ is not a
Sylow $p$-subgroup of $\gp{A^G}$, and let
$\overline G = G / \O_A(G)$.  Then $\overline A \ne 1$ and
\begin{equation}
\label{main-isomorphism}
\gp{\overline A^{\overline G}} = (L_1 \times L_2 \times \cdots
\times L_r)(D \cdot A_F)
\end{equation}
where $r \ge 1$, each $L_i$ is a simple group, and
$A_i = \overline A \cap L_i$ is a homocyclic abelian group.
Furthermore, $D = [D,A_F]$ is a (possibly trivial)
$p'$-group normalizing each $L_i$, and
$A_F$ is a (possibly trivial) abelian subgroup
of $\overline A$ of rank at most $r$ normalizing $D$ and
each $L_i$ and inducing outer
automorphisms on each $L_i$, and the extension
$(A_1 \cdots A_r):A_F$ splits.
Each $L_i$ belongs to one of the following families:

\begin{enumerate}
\item[(i)]
$L_i$ is a group of Lie type in characteristic $\ne p$ whose Sylow
$p$-subgroup is abelian but not elementary abelian; in this case
the Sylow $p$-subgroup of $L_i$ is homocyclic of the same rank as
$A_i$ but larger exponent than $A_i$;
here $D/(D \cap L_i C_{\overline G}(L_i))$ is a cyclic
$p'$-subgroup of the outer diagonal automorphism group
of $L_i$, and $A_F/C_{A_F}(L_i)$ acts
as a cyclic group of field automorphisms on $L_i$.

\item[(ii)]
$L_i \iso U_3(p^n)$ or $Re(3^n)$ is a group of $BN$-rank 1
($p=3$ with $n$ odd and $\ge 2$ in the latter family);
in the unitary case $A_i$ is the
center of a Sylow $p$-subgroup of $L_i$
(elementary abelian of order $p^n$),
and in the Ree group case $A_i$ is either the center or the
commutator subgroup of a Sylow 3-subgroup (elementary
abelian of order $3^n$ or $3^{2n}$ respectively);
in both families $D$ and $A_F$ act trivially on $L_i$.

\item[(iii)]
$L_i \iso G_2(q)$ with $(q,3) = 1$; here $\ord{A_i} = 3$
and both $D$ and $A_F$ act trivially on $L_i$.

\item[(iv)]
$L_i$ is one of the following sporadic groups, where in each case
$A_i$ has prime order, and both $D$ and $A_F$ act trivially on $L_i$:

\begin{description}
\item[($p = 3$)\hphantom{1}]
$J_2$,
\item[($p = 5$)\hphantom{1}]
$Co_3$, $Co_2$, $HS$, $Mc$,
\item[($p = 11$)]
$J_4$.
\end{description}

\item[(v)]
$L_i \iso J_3$, $p = 3$, and $A_i$ is either the center or the
commutator subgroup of a Sylow 3-subgroup (elementary
abelian of order 9 or 27 respectively); here $D$ and $A_F$ act
trivially on $L_i$.
\end{enumerate}
\end{theorem}

\noindent
{\bf Remark.}
After factoring out $\O_A(G)$ --- so that overbars may be omitted ---
the proof of Theorem~\ref{theorem3-1} shows that
$F^*(G) = L_1 \times \cdots \times L_r$, and
(\ref{main-isomorphism}) may also be written as
$$
\gp{A^G} \iso ((L_1 \times \cdots \times L_i)D \times
L_{i+1} \times \cdots \times L_j) A_F
\times (L_{j+1} \times \cdots \times L_r)
$$
where $L_1,\dots,L_i$ are the components of type $PSL$ or $PSU$
over fields of characteristic $\ne p$,
$L_{i+1},\dots,L_j$ are other groups listed in conclusion (i)
(but not linear or unitary),
and $L_{j+1},\dots,L_r$ are the components of types listed in (ii) to (v).
Furthermore, assume $G = \gp{A^G}$ and let
$A \le S \in Syl_p(G)$ and $S^* = S \cap F^*(G)$.
Then we may choose $D$ generically as
$[O_{p'}(C_G(S^*)),S]$, which is a $p'$-group normalized by $S$
and centralized by the Sylow $p$-subgroup
$S^*$ of $L_1 \cdots L_r$.

%
An easy example where both $D$ and $A_F$ are nontrivial is
provided at the outset of Section~\ref{examples}.

Conversely, observe that any finite group that has a
composition factor of one of the above types for $L_i$ possesses a
strongly closed $p$-subgroup that is not a Sylow $p$-subgroup of
its normal closure in $G$.
More detailed information about the structure of the Sylow $p$-subgroups
and their normalizers for the simple groups $L_i$ appearing in the
conclusion to this theorem is given from Proposition~\ref{sylow-structure}
through Corollary~\ref{normalizer-of-A-all} following.

Theorem~\ref{theorem3-1} is derived at the end of Section~\ref{theorem3-2-proof}
as a consequence of the next result, which is the minimal
configuration whose proof appears in Section~\ref{theorem3-2-proof}.

\begin{theorem}
\label{theorem3-2}
Assume the hypotheses of Theorem~\ref{theorem3-1}. Assume also
that $A$ is a minimal strongly closed subgroup of $G$, i.e., no
proper, nontrivial subgroup of $A$ is also strongly closed. Then
the conclusion of Theorem~\ref{theorem3-1} holds with the
additional results that $A$ is elementary abelian, $D = 1$,
$A_F = 1$, and $G$ permutes $L_1,\dots,L_r$
transitively (hence they are all isomorphic).
\end{theorem}

Some important consequences needed for our results
on cellularization of classifying spaces in \cite{Flores-Foote08} are the following.

\begin{corollary}
\label{corollary3-3}
Let $p$ be any prime, let $G$ be a finite group containing a
strongly closed $p$-subgroup $A$, let $S$ be a Sylow $p$-subgroup
of $G$ containing $A$, and
let $\overline G = G/\O_A(G)$.
Assume that $G$ is generated by the
conjugates of $A$.
Then $N_{\overline G}(\overline A)$ controls strong $\overline G$-fusion
in $\overline S$.
Furthermore, if $p \ne 3$ or if $\overline G$ does not have a
component of type $G_2(q)$ with $9 \divides q^2 -1$, then
$N_{\overline G}(\overline S)$ controls strong $\overline
G$-fusion in~$\overline S$.
\end{corollary}

In Section~\ref{exotic} we demonstrate that the exceptional case to
the stronger conclusion in the last sentence of
Corollary~\ref{corollary3-3} is unavoidable,
even if we impose the condition that $\Omega_1(S) \le A$:
we construct examples of groups $G$ generated by conjugates of
a strongly closed subgroup $A$ containing $\Omega_1(S)$ and
$G/\O_A(G) \iso G_2(q)$ where $N_{\overline G}(\overline S)$ does
not control fusion in $\overline S$.
%

The next result facilitates computation of $N_G(A)$ in groups satisfying
the conclusion to the preceding corollary.

\begin{corollary}
\label{corollary3-4}
Assume the hypotheses of preceding corollary and
the notation of Theorem~\ref{theorem3-1}.
For each $i$ let $C_i = C_{\overline G}(A_F) \cap N_{L_i}(A_i)$
and $S_i = \overline S \cap L_i$.  Then
$$
N_{\overline G}(\overline A)/\, \overline A
= (S_1 C_1 / A_1) \times
(S_2 C_2/ A_2)
\times \cdots \times
(S_r C_r/ A_r) .
$$
In particular, if $L_i$ is a component on which $A_F$ acts trivially
--- which is the case for all components in conclusions (ii) to (v)
of Theorem~\ref{theorem3-1} ---
the $i^{\text{th}}$ direct factor above may be replaced by just
$N_{L_i}(A_i)/A_i$ (and this applies to all factors if $A_F = 1$).
\end{corollary}


\medskip

The proof of Theorem~\ref{theorem3-2} relies on the Classification
of Finite Simple Groups. We reduce to the case where a minimal
counterexample, $G$, is a simple group having a strongly closed
$p$-subgroup $A$ that is properly contained in a non-abelian Sylow
$p$-subgroup $S$ of $G$. The remainder of the proof involves
careful investigation of the families of simple groups to
determine precisely when this happens.

We note that ``most'' simple groups do possess a strongly closed
$p$-subgroup that is proper in a Sylow $p$-subgroup, that is,
conclusion (i) of Theorem~\ref{theorem3-1} is
the ``generic obstruction'' in the
following sense. Let $\L_n(q)$ denote a simple group of Lie type
and $BN$-rank $n$ over the finite field $\FF_q$ with $(q,p) = 1$.
As we shall see in Section~\ref{preliminary-results},
for all but the finitely many primes dividing the order of
the Weyl group of the untwisted version of
$\L_n(q)$ the Sylow $p$-subgroups of $\L_n(q)$
are homocyclic abelian. Furthermore, the order of $\L_n(q)$ can be
expressed as a power of $q$ times factors of the form
$\Phi_m(q)^{r_m}$ for various $m, r_m \in \NN$, where $\Phi_m(x)$
is the $m^{\text{th}}$ cyclotomic polynomial. Then by
Proposition~\ref{sylow-structure} below,
if $m_0$ is the multiplicative order of $q$ (mod~$p$), then
$p$ divides $\Phi_{m_0}(q)$
and the abelian Sylow $p$-subgroup of $\L_n(q)$ is homocyclic of
rank $r_{m_0}$ and exponent $\ord{\Phi_{m_0}(q)}_p$. In particular it is
not elementary abelian whenever $p^2 \divides \Phi_{m_0}(q)$. For
example, this is the case in the groups $PSL_{n+1}(q)$ whenever $p
> n+1$ and $p^2$ divides $q^m - 1$ for some $m \le n+1$.  Thus for
fixed $n$ and all but finitely many $p$, this can always be
arranged by taking $q$ suitably large.

The overall organization of the paper is as follows: Section~2
contains preliminary results, including detailed information
on the Sylow structure and Sylow normalizers
of simple groups containing strongly closed $p$-subgroups.
The main results are proved in Section~3;
Theorem~\ref{theorem3-2} is proved first and
Theorem~\ref{theorem3-1} and its corollaries are derived
at the end of this section as consequences of it.
Section~4 provides interesting examples of groups, $G$,
possessing strongly closed subgroups, $A$; and with an eye to
applications in \cite{Flores-Foote08} we also describe $N_G(A)$ and
$N_G(S)$ for these cases of $G$.
More explicitly, we describe these first for $G$ simple,
and then for split extensions,
and finally for certain nonsplit extensions of simple groups.
The latter are very illuminating in the sense that they give an
alluring glimpse of what ``should be'' the $B\Z /p$-cellularization
of more general objects.

\section{Preliminary Results}
\label{preliminary-results}
The special case when $A$ has order $p$ has already been treated
in \cite[Proposition 7.8.2]{GLS3}. It is convenient to quote this
special case, although with extra effort our arguments could be
reworded to independently subsume it.

\begin{proposition}
\label{proposition3-4}
If $K$ is simple and $G = AK$ is a subgroup of $\Aut(K)$ such that
$A$ is strongly closed and $\ord A = p$, then $A \le K = G$ and
either the Sylow $p$-subgroups of $G$ are cyclic, or $G$ is
isomorphic to  $U_3(p)$ or one of the simple groups listed in
conclusions (iii) and (iv) of Theorem~\ref{theorem3-1}.
\end{proposition}

The authors of this result remark that an immediate consequence of
this is the odd-prime version of Glauberman's celebrated
$Z^*$-Theorem.
\begin{proposition}
\label{Zp-star-theorem}
If an element of odd prime order $p$ in any finite group $X$ does
not commute with any of its distinct conjugates then it lies in
$Z(X/O_{p'}(X))$.
\end{proposition}

We record some basic facts about strongly closed subgroups (the second of
which relies on the odd-prime $Z^*$-Theorem).

\begin{lemma}
\label{basic-facts}
For $p$ any prime let $A$ be a strongly closed $p$-subgroup of
$G$.
\begin{enumerate}

\item[(1)]
If $N$ is any normal subgroup of $G$ then $AN/N$ is a strongly
closed $p$-subgroup of $G/N$.

\item[(2)]
If $A$ normalizes a subgroup $H$ of $G$ with $O_{p'}(H) = 1$ and
$A \cap H = 1$ then $A$ centralizes $H$.
\end{enumerate}
\end{lemma}

\begin{proof}
In part (1) let $A \le S \in Syl_p(G)$. This result follows
immediately from the definition of strongly closed applied in the
Sylow $p$-subgroup $SN/N$ of $G/N$ together with Sylow's Theorem.
The proof of (2) is the same as for $p=2$ since, as noted earlier,
the $Z^*$-Theorem holds also for odd primes: by induction reduce
to the case where $G = AH$ and $C_A(H) = 1$. Then any element of
order $p$ in $A$ is isolated, hence lies in the center.
\end{proof}

The next few results gather facts about the simple groups appearing in
the conclusions to Theorems~\ref{theorem3-1-p-equal-2} and \ref{theorem3-1}.

The cross-characteristic Sylow structures of the simple groups of
Lie type are beautifully described in \cite[Section 10]{GL}
and reprised in \cite[Section 4.10]{GLS3}. Let $\L(q)$ denote a
universal Chevalley group or twisted variation over the field
$\FF_q$. (In the notation of \cite{GLS3}, $\L(q) = {}^d L(q)$,
where $d = 1,2,3$ corresponds to the untwisted, Steinberg twisted,
or Suzuki-Ree twisted variations respectively). Let $W$ denote the
Weyl group of the untwisted group corresponding to $\L(q)$. Except
for some small order exceptions, $\L(q)$ is a quasisimple group;
for example $\mathcal A_{\ell}(q) \iso SL_{\ell+1}(q)$ and
${}^2A_\ell(q) \iso SU_{\ell+1}(q)$. There is a set $\O(\L(q))$ of
positive integers, and ``multiplicities'' $r_m$ for each $m \in
O(\L(q))$, such that
$$
\ord{\L(q)} = q^N \prod_{m \in \O(\L(q))} (\Phi_m(q))^{r_m}
$$
where $\Phi_m(x)$ is the cyclotomic polynomial for the
$m^{\text{th}}$ roots of unity.

Let $p$ be an odd prime not dividing $q$ and assume $S$ is a
nontrivial Sylow $p$-subgroup of $\L(q)$.
Let $m_0$ be the smallest element of $\O(\L(q))$ such that
$p \divides \Phi_{m_0}(q)$.
Let
\begin{equation}
\label{weyl-order}
\mathcal W = \{ m \in \O(\L(q)) \mid m = p^a m_0, \ a \ge 1 \}
\qquad \text{and} \qquad
b = \sum_{m \in \mathcal W} r_m
\end{equation}
where $b = 0$ if $\mathcal W = \emptyset$.
The main structure theorem is as follows.

\begin{proposition}
\label{sylow-structure}
Under the above notation the following hold:
\begin{enumerate}

\item[(1)]
$m_0$ is the multiplicative order of $q \pmod p$.

\item[(2)]
Except in the case where $\L(q) = {}^3D_4(q)$ with $p=3$, $S$ has
a nontrivial normal homocyclic subgroup, $S_T$, of rank $r_{m_0}$ and
exponent $\ord{\Phi_{m_0}(q)}_p$.

\item[(3)]
With the same exception as in (2), $S$ is a split extension of
$S_T$ by a (possibly trivial) subgroup $S_W$ of order $p^b$
(where $b$ is defined in (\ref{weyl-order})),
and $S_W$ is isomorphic to a subgroup of $W$.
In particular, if $p \notdivides \ord W$ or if $pm_0 \notdivides m$
for all $m \in \O(\L(q))$, then $S = S_T$ is homocyclic abelian.

\item[(4)]
If $\L(q) = {}^3D_4(q)$ with $p=3$ and $\ord{q^2 - 1}_3 = 3^a$,
then $S$ is a split extension of an abelian group of type
$(3^{a+1},3^a)$ by a group of order 3, and $S$ has rank 2.

\item[(5)]
If $\L(q)$ is a classical group (linear, unitary, symplectic or
orthogonal) then every element of order $p$ is conjugate to some
element of $S_T$.

\item[(6)]
Except in ${}^3 D_4(q)$ (where $S_W$ is not defined),
$S_W$ acts faithfully on $S_T$; and
in the simple group $\L(q)/Z(\L(q)) = \overline{\L(q)}$
we have $\overline{S_W} \iso S_W$ acts faithfully on $\overline{S_T}$
except when $p=3$ with
$\L(q) \iso SL_3(q)$ (with $3 \divides q-1$ but $9 \notdivides q-1$)
or $SU_3(q)$ (with $3 \divides q+1$ but $9 \notdivides q+1$).

\item[(7)]
If a Sylow $p$-subgroup of the simple group
$\L(q)/Z(\L(q))$ is abelian but not elementary abelian
then $p$ does not divide
the order of the Schur multiplier of $\L(q)$.
\end{enumerate}
\end{proposition}

\begin{proof}
For parts (1) to (6)
see \cite[10-1, 10-2]{GL} or \cite[Theorems 4.10.2, 4.10.3]{GLS3}.
If the odd prime $p$ divides the order of the Schur multiplier of
$\L(q)$ then by \cite[Table 6.12]{GLS3} we must have
$\L(q)$ of type $SL_n(q)$, $SU_n(q)$, $E_6(q)$ or ${}^2 E_6(q)$ with
$p$ dividing $(n,q-1)$, $(n,q+1)$, $(3,q-1)$ or $(3,q+1)$ respectively.
It follows easily from (6) that in each of the
corresponding simple groups a Sylow
$p$-subgroup cannot be abelian of exponent $\ge p^2$.
\end{proof}

We shall frequently adopt the efficient shorthand from the
sources just cited for the latter families.

\medskip\noindent
{\bf Notation.}
Denote $SL_n(q)$ by $SL_n^+(q)$ and $SU_n(q)$ by $SL_n^-(q)$ (likewise for the
general linear and projective groups);
and say a group is of type $SL_n^\epsilon(q)$ according
to whether $p \divides q - \epsilon$ for $\epsilon = +1$, $-1$ respectively
(dropping the ``1'' from $\pm 1$).
The analogous convention is adopted for $E_6(q) = E_6^+(q)$
and ${}^2 E_6(q) = E_6^-(q)$.

\medskip

The following general result is especially important for the
groups of Lie type.

\begin{proposition}
\label{irreducible-action}
If $G$ is any simple group with an abelian Sylow $p$-subgroup
$S$ for any prime $p$,
then $N_G(S)$ acts irreducibly and nontrivially on $\Omega_1(S)$,
and so $S$ is homocyclic.
In particular, a nontrivial subgroup
of $S$ is strongly closed if and only if it is homocyclic of the
same rank as $S$.
\end{proposition}

\begin{proof}
See \cite[Proposition 7.8.1]{GLS3} and \cite[12-1]{GL}.
\end{proof}

\begin{proposition}
\label{normalizer-of-A-Lie}
Let $G$ be a simple group of Lie type over $\FF_q$
and let $p$ be an odd prime not dividing $q$.
Assume a Sylow $p$-subgroup $S$ of $G$ is abelian and let $A = \Omega_1(S)$.
Then $N_G(A) = N_G(S)$.
\end{proposition}

\begin{proof}
The result is trivial if $S = A$ so assume this is not the case;
in particular the exponent of $S$ is at least $p^2$.
By part (7) of Proposition~\ref{sylow-structure}, $p$ does not divide the order
of the Schur multiplier of $G$, so we may assume $G$ is the (quasisimple)
universal cover of the simple group.
Clearly $N_G(S) \le N_G(A)$.
Moreover, since $S \in Syl_p(C_G(A))$, by Frattini's Argument
$N_G(A) = C_G(A)N_G(S)$.
Thus it suffices to show $C_G(A) = C_G(S)$.
Since $C_G(A)$ has an abelian Sylow $p$-subgroup and since
any nontrivial $p'$-automorphism of $S$ must act nontrivially on
$A$, by Burnside's Theorem $C_G(A)$ has a normal $p$-complement.
Let $\Delta = [O_{p'}(C_G(A)),S]$.
It suffices to prove $S$ centralizes $\Delta$.

Let $\overline G$ be the simply connected universal
algebraic group over the algebraic closure of $\FF_q$, and let
$\sigma$ be a Steinberg endomorphism whose fixed points equal $G$.
In the notation of Proposition~\ref{sylow-structure}, since $S = S_T$,
by the proof of \cite[Theorem 4.10.2]{GLS3} there is a
$\sigma$-stable maximal torus $\overline T$
of $\overline G$ containing $S$.
Let $\overline C$ denote the connected component of
$C_{\overline G}(\overline A)$, so $\overline C$ is also
$\sigma$-stable.
Note that $\overline T \le \overline C$ and since
$\Delta$ is generated by conjugates of $S$, so too
$\Delta \le \overline C$.
We may now follow the basic ideas in the proof of
\cite[Theorem 7.7.1(d)(2)]{GLS3}, where more background is provided.
By \cite[4.1(b)]{SS}, $\overline C$ is reductive, so by
the general theory of connected reductive groups
$$
\overline C = \overline Z \, \overline L
$$
where $\overline Z$ is the connected component of the center of
$\overline C$, $\overline L$ is the semisimple component
(possibly trivial), and $\overline Z \cap \overline L$ is a finite group.
Since $\Delta \le \overline C'$ we have $\Delta \le \overline L$.
The group of fixed points of $\sigma$ on $\overline L$ is a
commuting product $L_1 \cdots L_n$ of (possibly solvable)
groups of Lie type over the same characteristic as $G$ and smaller rank,
and $S$ induces inner or diagonal automorphisms on each $L_i$.
Since $\Delta \le O_{p'}(C_G(A))$ we have
$$
\Delta \le O_{p'}(L_1 \cdots L_n) = O_{p'}(L_1) \cdots O_{p'}(L_n).
$$
If $L_i$ is a $p'$-group, then $\text{Inndiag}(L_i)$ is also
a $p'$-group and so $S$ centralizes $L_i$.  On the other hand,
if $p$ divides the order of $L_i$, then $O_{p'}(L_i) \le Z(L_i)$;
in this case $\text{Inndiag}(L_i)$ centralizes $Z(L_i)$.
In all cases $S$ centralizes $O_{p'}(L_i)$, as needed.
\end{proof}

\begin{proposition}
\label{normalizer-of-others-in-theorem3}
Let $p$ be any prime, let $G$ be a simple group containing a strongly
closed $p$-subgroup, let $S \in Syl_p(G)$ and let $Z = Z(S)$.

\begin{enumerate}

\item[(1)]
Assume $G \iso U_3(q)$ with $q = p^n$, or $G \iso Sz(q)$ with $p = 2$ and $q = 2^n$.
Then $S$ is a special group of type $q^{1+2}$ or $q^{1+1}$ respectively,
and $N_G(S) = N_G(Z) = SH$, where the Cartan subgroup $H$ is
cyclic of order $(q^2-1)/(3,q+1)$ or $q-1$ respectively.
In both families $H$ acts irreducibly on both $Z$ and $S/Z$,
and $Z$ is the unique
nontrivial, proper strongly closed subgroup of $S$.

\item[(2)]
Assume $G \iso Re(q)$ with $p = 3$ and $q = 3^n$, $n > 1$.
Then $S$ is of class 3, $Z \iso E_q$ and
$S' = \Phi(S) = \Omega_1(S) \iso E_{q^2}$.
Furthermore, $N_G(S) = N_G(Z) = SH$, where the Cartan subgroup $H$ is
cyclic of order $q-1$ and acts irreducibly on all three
central series factors:
$Z$ and $S'/Z$ and $S/S'$.
Thus $Z$ and $\Omega_1(S)$ are the only nontrivial
proper strongly closed subgroups of $S$.

\item[(3)]
Assume $G \iso G_2(q)$ for some $q$ with $(q,3) = 1$ and $p = 3$.
Then $Z \iso Z_3$ is the only nontrivial
proper strongly closed subgroup of $S$.
Furthermore, $N_G(Z) \iso SL_3^\epsilon(q)\cdot 2$
according to whether $3 \divides q-\epsilon$.
An element of order 2 in $N_G(Z) - C_G(Z)$ inverts $Z$, and
$N_G(S)/S \iso QD_{16}$ or $E_4$ according as $\ord S = 3^3$ or
$\ord S > 3^3$ respectively.
No automorphism of $G$ of order 3 normalizes $S$ and
centralizes both $S/Z$ and a $3'$-Hall subgroup of $N_G(S)$.

\item[(4)]
Assume $G$ is isomorphic to one of the sporadic groups: $J_2$ (with $p=3$);
$Co_2$, $Co_3$, $HS$, $Mc$ (with $p=5$); or $J_4$ (with $p=11$).
In each case $S$ is non-abelian of order $p^3$ and exponent $p$,
and $Z$ is the only nontrivial
proper strongly closed subgroup of $S$.
The normalizer of $Z$ [in $G$] is:
$3PGL_2(9)$ [in $J_2$],
$5^{1{+}2}((4*SL_3(3))\cdot 2)$ [in $Co_2$],
$5^{1{+}2}((4 \text{Y} S_3)\cdot 4)$ [in $Co_3$],
$5^{1{+}2}(8\cdot 2)$ [in $HS$],
$(5^{1{+}2}\cdot 3)\cdot 8$ [in $Mc$], or
$(11^{1{+}2}\cdot SL_2(3))\cdot 10$ [in $J_4$].
In $G = J_2$ we have $N_G(S)/S \iso Z_8$; and in all other cases
$N_G(S) = N_G(A)$.

\item[(5)]
Assume $G \iso J_3$ with $p=3$.
Then $Z \iso E_9$ and $\Omega_1(S) \iso E_{27}$ are the only nontrivial
proper strongly closed subgroups of $S$.
Furthermore, $N_G(Z) = N_G(S) = SH$ where
$H \iso Z_8$ acts fixed point freely on
$\Omega_1(S)$ and irreducibly on $Z$.

\end{enumerate}
\end{proposition}

\begin{proof}
Part (1) may be found in \cite{HKS72} and \cite{Suz62}.
Part (2) appears in \cite{Wa66}.
All parts of (4) and (5) appear in \cite[Chapter~5]{GLS3}
with references therein.

In part (3), by \cite[14-7]{GL} the center of $S$ has order 3 and
$C = C_G(Z) \iso SL_3^\epsilon(q)$
according to the condition $3 \divides q - \epsilon$.
The same reference shows
$G$ has two conjugacy classes of elements of order 3:
the two nontrivial elements of $Z$ are in one class, and all
elements of order 3 in $S - Z$ lie in the other.
Now $S \le SL_3^\epsilon (q)$ acts absolutely irreducibly on
its natural 3-dimensional module over $\FF_q$
(or $\FF_{q^2}$ in the unitary case),
hence by Schur's Lemma the centralizer of $S$ in $C$
consists of scalar matrices.
Thus $Z = C_C(S) = C_G(S)$.
Since the two nontrivial elements of $Z$ are conjugate in $G$,
$N_G(Z) = C\gp t$ where an involution $t$ may be chosen to normalize
$S$ and induce a graph (transpose-inverse) automorphism on $C$.
By canonical forms, all non-central elements of order 3 in
$SL_3^\epsilon(q)$ are conjugate in $GL_3^\epsilon(q)$
to the same diagonal matrix
$u = \text{diag}(\lambda,\lambda^{-1},1)$,
where $\lambda$ is a primitive cube root of unity,
but are also conjugate in $SL_3^\epsilon(q)$ to $u$
because the outer (diagonal) automorphism group induced
by $GL_3^\epsilon$ may be represented by diagonal matrices that
commute with $u$.
Thus all elements of order 3 in $S - Z$ are conjugate in $C$.

If $\ord S = 27$, then since $S/Z$ is abelian of type (3,3),
all elements of order 3
in $S/Z$ are conjugate under the action of $N_C(S/Z) = N_C(S)/Z$;
hence they are conjugate under the faithful action of
a $3'$-Hall subgroup, $H_0$, of $N_C(S)$ on $S/Z$.
This shows $\ord{H_0} \ge 8$.
Since a $3'$-Hall subgroup $H$ of
$N_G(S)$ acts faithfully on $S/Z$ and has order $2 \ord{H_0}$,
it must be isomorphic to a Sylow 2-subgroup, $QD_{16}$, of $GL_2(3)$
as claimed.

If $\ord S = 3^{2a+1} > 27$
then we may describe $S$ as the group, $S_T$, of diagonal
matrices of 3-power order acted upon by a permutation matrix $w$ of order 3
(where $\gp w = S_W$).
Then $S_T \iso Z_{3^a} \times Z_{3^a}$
is the unique abelian subgroup of $S$ of index 3 (as $\ord Z = 3$),
so $N_C(S)$ normalizes $S_T$.  Let $H_0$ be a $3'$-Hall subgroup of $N_C(S)$.
One easily sees that $H_0$ must act faithfully on $\Omega_1(S_T)$
(and centralize $Z$), hence $\ord{H_0} \le 2$.
Since there is a permutation matrix of order 2 in $C$ normalizing $S$,
$\ord{H_0} = 2$. Thus $N_G(S)/S$ has order 4, and is seen to be a fourgroup
by its action on $\Omega_1(S_T)$.

To see that $Z$ is the unique nontrivial strongly closed subgroup that
is proper in $S$ suppose $B$ is another, so that $Z < B$.
If $B$ contains an element of order 9 --- hence an element of order 9
represented by a diagonal matrix in $C$ --- then by conjugating in $C$ one
easily computes that $B - Z$ contains an element of order 3.
Since all such are conjugate in $C$ this shows $\Omega_1(S) \le B$.
It is an exercise that $\Omega_1(S) = S$ (the details appear at the
end of the proof of Lemma~\ref{G-not-classical}), a contradiction.

Finally, suppose $f$ is an automorphism of $G$ of order 3 that
normalizes $S$ and centralizes $S/Z$.
Then $\ind{S}{C_S(f)} \le 3$ so $f$ cannot be a field automorphism
as $\ind{G_2(r^3)}{G_2(r)}_3 \ge 3^2$ for all $r$ prime to 3.
Thus $f$ must induce an inner automorphism on $G$, hence act as
an element of order 3 in $S_T$.
We have already seen that no such element centralizes a
$3'$-Hall subgroup of $N_G(S)$, a contradiction.
This completes all parts of the proof.
\end{proof}

\begin{corollary}
\label{normalizer-of-A-all}
Let $p$ be any prime, let $L$ be a finite simple group possessing a
strongly closed $p$-subgroup $A$ that is properly contained in
the Sylow $p$-subgroup $S$ of $L$.
Assume further that $L$ is isomorphic to one of the groups $L_i$ in
the conclusion of Theorem~\ref{theorem3-1-p-equal-2}
or Theorem~\ref{theorem3-1}.
Then one of the following holds:

\begin{enumerate}

\item[(1)]
$N_L(S) = N_L(A)$,

\item[(2)]
$\ord A = 3$ and $L \iso G_2(q)$ for some $q$ with $(q,3) = 1$, or

\item[(3)]
$\ord A = 3$, $L \iso J_2$ and
$N_L(A) \iso 3PGL_2(9)$.

\end{enumerate}
\end{corollary}

\begin{proof}
This is immediate from Propositions~\ref{normalizer-of-A-Lie}
and \ref{normalizer-of-others-in-theorem3}.
\end{proof}


\section{The Proofs of the Main Theorems}
\label{theorem3-2-proof}
%
%


In this section we first prove Theorem~\ref{theorem3-2};
Theorem~\ref{theorem3-1} and its corollaries are then derived from it
at the end of this section.

\subsection{The Proof of Theorem~\ref{theorem3-2}}
\label{theorem3-2-proof-subsection}
%
\    

Throughout this subsection $p$ is an odd prime,
$G$ is a minimal counterexample to Theorem~\ref{theorem3-2},
and $A$ is a nontrivial strongly closed subgroup of $G$ that is a
proper subgroup of the Sylow $p$-subgroup $S$ of $G$. The
minimality implies that if $H$ is any proper section of $G$
containing a nontrivial minimal strongly closed (with respect to
$H$) $p$-subgroup $A_0$, then either $A_0$ is a Sylow subgroup of
its normal closure in $H$ or the normal closure of
$\overline{A_0}$ in $\overline H$ is a direct product of
isomorphic simple groups, as described in the conclusion of
Theorem~\ref{theorem3-1}, where overbars denote passage to
$H/\O_{A_0}(H)$. In particular, $A_0$ does not even have to be a
subgroup of $A$, although for the most part we will be applying
this inductive assumption to subgroups $A_0 \le A \cap H$ (which
we often show is nontrivial by invoking part (2) of
Lemma~\ref{basic-facts}).

Familiar facts about the families of simple groups, including the
sporadic groups, are often stated without reference. All of these
can be found in the excellent, encyclopedic source \cite{GLS3}.
Specific references are cited for less familiar results that are
crucial to our arguments.

\bigbreak

\begin{lemma}
\label{G-is-simple}
$G$ is a simple group.
\end{lemma}

\begin{proof}
Since strong closure inherits to quotient groups,
if $\O_A(G) \ne 1$ we may apply induction to $G/\O_A(G)$
and see that the asserted
conclusion holds. Thus we may assume $\O_A(G) = 1$, i.e.,
\begin{equation}
\text{$A \cap N$ is not a Sylow $p$-subgroup of $N$ for any
nontrivial $N \nor G$.} \label{simple-1}
\end{equation}
In particular,
\begin{equation}
O_{p'}(G) = 1. \label{simple-2}
\end{equation}

Let $G_0 = \gp{A^G}$ and assume $G_0 \ne G$. By (\ref{simple-1}),
$A$ is not a Sylow $p$-subgroup of $G_0$. Let $1 \ne A_0 \le A$ be
a minimal strongly closed subgroup of $G_0$. By the inductive
hypothesis $A_0$ is contained in a semisimple normal subgroup $N$
of $G_0$ satisfying the conclusions of the theorem. Since $N \nor
G$ it follows that $M = \gp{N^G}$ is a semisimple normal subgroup
of $G$ whose simple components are described by
Theorem~\ref{theorem3-1}. Since $A$ is minimal strongly closed in
$G$ and $1 \ne A_0 \le A \cap M$, $A \le M$ and the conclusion of
Theorem~\ref{theorem3-2} is seen to hold. Thus
\begin{equation}
\text{$G$ is generated by the conjugates of $A$.} \label{simple-3}
\end{equation}
By strong closure $A \cap O_p(G) \nor G$, hence by
(\ref{simple-1}), $A \cap O_p(G) = 1$.  Thus $[A,O_p(G)] \le A
\cap O_p(G) = 1$, i.e., $A$ centralizes $O_p(G)$.  Since $G$ is
generated by conjugates of $A$,
\begin{equation}
O_p(G) \le Z(G). \label{simple-4}
\end{equation}
By (\ref{simple-2}) and (\ref{simple-4}), $F^*(G)$ is a product of
commuting quasisimple components, $L_1,\dots,L_r$, each of which has a
nontrivial Sylow $p$-subgroup. Since $A$ acts faithfully on
$F^*(G)$, by Lemma~\ref{basic-facts} $A \cap F^*(G) \ne 1$.
The minimality of $A$ then forces $A \le F^*(G)$. Thus $A$ normalizes
each $L_i$, whence so does $G$ by (\ref{simple-3}). Now $A$ acts
nontrivially on one component, say $L_1$, so again by
Lemma~\ref{basic-facts}, $A \cap L_1 \ne 1$. By minimality of $A$
we obtain $A \le L_1 \nor G$, so by (\ref{simple-3})
$$
\text{$G = L_1$ is quasisimple (with center of order a power of
$p$).}
$$

Finally, assume $Z(G) \ne 1$ and let $\widetilde G = G/Z(G)$.
Since $A \ne S$ but $A \cap Z(G) = 1$, by Gasch\"utz's Theorem we
must have that $S \ne AZ(G)$ and so $\widetilde A$ is strongly
closed but not Sylow in the simple group $\widetilde G$. Since
$\ord{\widetilde G} < \ord G$, the pair $(\widetilde G,\widetilde
A)$ satisfy the conclusions of Theorem~\ref{theorem3-2}; in
particular, $\widetilde A =  \Omega_1(Z(\widetilde S))$ in all
cases. If $\widetilde G$ is a group of Lie type in conclusion (i),
again by Gasch\"utz's Theorem together with the irreducible action
of $N_{\widetilde G}(\widetilde S)$ on $\Omega_1(\widetilde S)$,
$\widetilde A$ must lift to a non-abelian group in $G$. In this
situation $Z(G) \le A'$, contrary to $A \cap Z(G) = 1$. In
conclusions (ii), (iii) and (iv) the $p$-part of the multipliers
of the simple groups are all trivial, so $Z(G) = 1$ in these
cases. In case (v) when $\widetilde G \iso J_3$ and $\widetilde A
= Z(\widetilde S)$ by the fixed point free action of an element of
order 8 in $N_G(S)$ on $S$ it again follows easily that $\tilde A$
must lift to the non-abelian group of order 27 and exponent 3 in
$G$, contrary to $A \cap Z(G) = 1$. This shows $Z(G) = 1$ and so
$G$ is simple.  The proof is complete.
\end{proof}

\begin{lemma}
\label{A-is-noncyclic}
$A$ is not cyclic and $S$ is non-abelian.
\end{lemma}

\begin{proof}
If $A$ is cyclic then since $\Omega_1(A)$ is also strongly closed,
the minimality of $A$ gives that $\ord A = p$.
Then $G$ is not a
counterexample by Proposition~\ref{proposition3-4}.
Likewise if
$S$ is abelian, by Proposition~\ref{irreducible-action} it is
homocyclic with $N_G(S)$ acting irreducibly and nontrivially on
$\Omega_1(S)$. By minimality of $A$ we must then have
$A = \Omega_1(S)$ and the exponent of $S$ is greater than $p$.
None of the sporadic or alternating groups or groups of Lie type in
characteristic $p$ contain such Sylow $p$-subgroups, so $G$ must
be a group of Lie type in characteristic $\ne p$. Again, $G$ is
not a counterexample, a contradiction.
\end{proof}

Note that because $A$ is a noncyclic normal
subgroup of $S$ and $p$ is odd, $A$ contains an abelian subgroup
$U$ of type $(p,p)$ with $U \nor S$. Furthermore,
$\ind{S}{C_S(U)} \le p$ so $U$ is contained in
an elementary abelian subgroup of $S$ of maximal rank.

Lemmas~\ref{G-not-alternating} to \ref{G-not-sporadic} now
successively eliminate the families of simple groups as
possibilities for the minimal counterexample.  The argument used
to eliminate the alternating groups is a prototype for the more
complicated situation of Lie type groups, so slightly more
expository detail is included.

\begin{lemma}
\label{G-not-alternating}
$G$ is not an alternating group.
\end{lemma}

\begin{proof}
Assume $G \iso A_n$ for some $n$.  Since $S$ is non-abelian, $n
\ge p^2$. If $p \notdivides n$ then  $S$ is contained in a
subgroup isomorphic to $A_{n-1}$, which contradicts the minimality
of $G$ (no alternating group satisfies the conclusions in
Theorem~\ref{theorem3-1}). Thus $n = ps$ for some $s \in \NN$ with
$s \ge p$.

Let $E$ be a subgroup of $S$ be generated by $s$ commuting
$p$-cycles. Since $E$ contains a conjugate of every element of
order $p$ in $G$, $A \cap E \ne 1$. We claim $E \le A$.
Let $z = z_1 \cdots z_r \in A \cap E$
be a product of commuting $p$-cycles
$z_i$ in $E$ with $r$ minimal. If $r \ge 3$ there is an element
$\sigma \in A_n$ that inverts both $z_r$ and $z_{r-1}$ and
centralizes all other $z_i$; and if $r = 2$, since $n \ge 3r$
there is an element $\sigma \in A_n$ that inverts $z_2$ and
centralizes $z_1$. In either case, by strong closure
$z^\sigma \in A \cap E$ and
$zz^\sigma = z_1^2 \cdots z_{r-2}^2$ or $z_1^2$ respectively.
Hence $zz^\sigma$ is an element of $A \cap E$ that
is a product of fewer commuting $p$-cycles, a contradiction.
This shows $A$ contains a $p$-cycle, hence by strong closure $E \le A$.
Now $A_n$ contains a subgroup $H$ with
\begin{equation}
S \le H = N_{A_n}(E) \qquad \text{and} \qquad H \iso Z_p \wr A_s.
\label{alt-1}
\end{equation}
By our inductive assumption $H$ contains a normal subgroup
$N = \O_A(H)$ with $E \le N$
such that $A \cap N$ is a Sylow $p$-subgroup of $N$ and
$H/N$ a product of simple components described in
Theorem~\ref{theorem3-1}. Since $H$ is a split extension over $E$
and every element of $H$ of order $p$ is conjugate to an element
of $E$, by strong closure $A \ne E$. Since $H/E \iso A_s$ is not
one of the simple groups in Theorem~\ref{theorem3-1} it follows
that $N = H$ (in the cases where $s = 3$ or 4 as well), contrary
to $A \ne S$. This contradiction establishes the lemma.

Alternatively, one could argue from (\ref{alt-1}) and induction
that $S = \Omega_1(S)$, and so again $S = A$ by strong closure, a
contradiction.
\end{proof}



\begin{lemma}
\label{G-not-classical}
$G$ is not a classical group (linear, unitary, symplectic,
orthogonal) over $\FF_q$, where $q$ is a prime power not divisible
by $p$.
\end{lemma}

\begin{proof}
Assume $G$ is a classical simple group. Following the
notation in \cite[Theorem 4.10.2]{GLS3},
let $V$ be the classical vector space
associated to $G$ and let $X = \isom(V)$.
We may assume $\dim V \ge 7$ in the orthogonal case because of
isomorphisms of lower dimensional orthogonal groups with other
classical groups (the dimension is over $\FF_{q^2}$ in the unitary case).
The tables in \cite[Chapter 4]{KL} are helpful references in this
proof.

First consider when $G$ is neither a linear group with $p$
dividing $q-1$ nor a unitary group
with $p$ dividing $q+1$.
This restriction implies that
$p \notdivides \ind{X}{X'}$ and there is
a surjective homomorphism $X' \rightarrow G$ whose kernel is a
$p'$-group. Thus we may do calculations in $X$ in place of $G$
(taking care that conjugations are done in $X'$).
Proposition~\ref{sylow-structure} is realized explicitly in
this case as follows: There is a decomposition
$$
V = V_0 \perp V_1 \perp \cdots \perp V_s
$$
of $V$ ($\perp$ denotes direct sum in the linear case), where
$\isom(V_0)$ is a $p'$-group, the cyclic group of order $p$ has an
orthogonally indecomposable representation on each other $V_i$,
the $V_i$ are all isometric, and a Sylow $p$-subgroup of
$\isom(V_i)$ is cyclic. Furthermore, $X'$ contains a subgroup
isomorphic to $A_s$ permuting $V_1,\dots,V_s$ and the stabilizer
in $X'$ of the set $\{V_1,\dots,V_s\}$ contains a Sylow
$p$-subgroup of $X$. In other words, we may assume
\begin{equation}
S \le H \iso \isom(V_1) \wr A_s. \label{classical-1}
\end{equation}
In the notation of Proposition~\ref{sylow-structure},
let $S \cap \isom(V_i) = \gp{u_i}$,
where $u_i$ acts trivially on $V_j$ for all $j \ne i$.
Then $S_T = \gp{u_1,\dots,u_s}$ and $S_W$ is
a Sylow $p$-subgroup of $A_s$.
Since $S$ is non-abelian,
$S_W \ne 1$ and so $s \ge p \ge 3$. Let $z_i$ be an element of
order $p$ in $\gp{u_i}$, and let
$$
E = \gp{z_1,\dots,z_s} = \Omega_1(S_T) \iso E_{p^s}.
$$
The faithful action of $S_W$ on $S_T$ forces $Z(S) \le S_T$, so $A
\cap E \ne 1$.

We claim $E \le A$. As in the alternating group case, let $z$ be a
nontrivial element in $A \cap E$ belonging to the span of $r$ of
the basis elements $z_i$ in $E$ with $r$ minimal. After
renumbering and replacing each $z_i$ by another generator for
$\gp{z_i}$ if necessary, we may assume $z = z_1 \cdots z_r$.
If $r \ge 3$ there is an element $\sigma \in G$ that acts trivially on
$z_1,\dots,z_{r-2}$ and normalizes but does not centralize
$\gp{z_{r-1},z_r}$; and if $r =2$, since $s \ge 3$ there is an
element $\sigma \in G$ that centralizes $z_1$ and normalizes but
does not centralize $\gp{z_2}$. In both cases $z^\sigma z^{-1}$ is
a nontrivial element of $A \cap E$ that is a product of fewer
basis elements.
This shows $z_i \in A$ for some $i$ and so $E \le A$
since all $z_j$ are conjugate in $G$.

By Proposition~\ref{sylow-structure}(5) in this setting, every
element of order $p$ in $G$ is conjugate to an element of $E$.
Since the extension in (\ref{classical-1}) is split, $A \not\le
S_T$. By the overall induction hypothesis applied in $H$ (or
because a Sylow $p$-subgroup of $A_s$ is generated by elements
of order $p$), it follows that $A$ covers $S/S_T$. We may
therefore choose a numbering so that for some $x \in A$, $u_1^x =
u_2$.  Thus
$$
u = u_1 u­_2^{-1} = [u_2,x] \in
A \cap \isom(V_1 \perp \cdots \perp V_{s-1}).
$$
Let $Y = G \cap \isom(V_1 \perp \cdots \perp V_{s-1})$ so that $Y$
is also a classical group of the same type as $G$ over $\FF_q$.
Note that the dimension of the underlying space on which $Y$ acts
is at least $2(s-1)$ by our initial restrictions on $q$. Since $Y$
is proper in $G$, by induction applied using a minimal strongly
closed subgroup $A_0$ of $A \cap Y$ in $Y$ we obtain the
following: either $A_0$ (hence also $A$) contains a Sylow
$p$-subgroup of $Y$, or the Sylow $p$-subgroups of $Y$ are
homocyclic abelian with $A_0 \cap Y$ elementary abelian of the
same rank as a Sylow $p$-subgroup of $Y$. Furthermore, in the
latter case a Sylow $p$-normalizer acts irreducibly on $A_0$, and
hence the strongly closed subgroup $A \cap Y$ is also homocyclic
abelian. Since $A \cap Y$ contains the element $u$ of order $d$,
where $d = \ord{u_1}$, in either case $A \cap Y$ contains all
elements of order $d$ in $S \cap Y$. Since $u_1 \in S \cap Y$ this
proves $u_1 \in A$. By (\ref{classical-1}) all $u_i$ are conjugate
in $G$ to $u_1$, hence $S_T \le A$ and so $A = S$ a contradiction.

It remains to consider the cases where $V$ is of linear or unitary
type and $p$ divides $q-1$ or $q+1$ respectively (denoted as usual
by $p \divides q-\epsilon$). Now replace the simple group $G$ by
its universal quasisimple covering $SL^\epsilon(V)$. Likewise
replace $A$ by the $p$-part of its preimage. Thus $A$ is a
noncyclic (hence noncentral) strongly closed $p$-subgroup of
$SL^\epsilon(V)$. In this situation $S = S_T S_W$ where we may
assume $S_T$ is the group of $p$-power order diagonal matrices of
determinant 1 (over $\FF_{q^2}$ in the unitary case), and $S_W$ is
a Sylow $p$-subgroup of the Weyl group $W$ of permutation matrices
permuting the diagonal entries. Furthermore, $S_T$ is homocyclic
of exponent $d$, where $d = \ord{q - \epsilon}_p$, and is a trace
0 submodule of the natural permutation module for $W$ of exponent
$d$ and rank $m = \dim V$. Since $A$ is noncyclic, it contains a
noncentral element $z$ of order $p$; and by
Proposition~\ref{sylow-structure}, $z$ is conjugate to an element
of $S_T$, i.e., is diagonalizable. Arguing as above with $E =
\Omega_1(S_T)$ we reduce to the case where $z$ is represented by
the matrix $\text{diag}(\zeta,\zeta^{-1},1,\dots,1)$ for some
primitive $p^{\text{th}}$ root of unity $\zeta$. The action of $W$
again forces $E \le A$. Again, every element of order $p$ in $S$
is conjugate in $G$ to an element of $E$, so by strong closure
\begin{equation}
\Omega_1(S) \le A. \label{classical-2}
\end{equation}

Consider first when $m \ge 5$.
Then $C_G(z)$ contains a quasisimple component
$L \iso SL_{m-2}^\epsilon(q)'$.
Since $L$ contains a conjugate of $z$, the
inductive argument used in the general case shows that $A \cap L$
contains a diagonal matrix element of order $d$, hence contains such an
element centralizing an $n-2$ dimensional subspace.
The strong closure of $A$ then again yields $S_T \le A$; and
as before by induction or because $S = S_T \Omega_1(S)$
we get $A = S$, a contradiction.

Thus $\dim V \le 4$, and since $S_W \ne 1$ we must have $p = 3$.
If $G \iso SL_4^\epsilon(q)$ then
let $z$ be represented by the diagonal matrix
$\text{diag}(\zeta,\zeta,\zeta,1)$, where $\zeta$ is a primitive
$3^{\text{rd}}$ root of unity.
Then $C_G(z)$ contains a Sylow 3-subgroup of $G$ and a
component of type $SL_3^\epsilon(q)$, so the preceding argument
leads to a contradiction.

Finally, consider when $G \iso SL_3^\epsilon(q)$.
The Sylow 3-subgroups of $SL_3^\epsilon(q)$
are described in the proof of
Proposition~\ref{normalizer-of-others-in-theorem3}.
In both instances $S_T$ is
homocyclic of rank 2 and exponent $d$ with generators $u_1,u_2$,
and with $S_W = \gp w \iso Z_3$ acting by
$$
u_1^w = u_2 \quad \text{and} \quad u_2^w = u_1^{-1} u_2^{-1}.
$$
Thus $u_1 w$ has order 3, and so $u_1 = (u_1 w)w^{-1} \in
\Omega_1(S)$. By (\ref{classical-2}), this again forces $A = S$,
which gives the final contradiction.
\end{proof}


\begin{lemma}
\label{G-not-exceptional}
$G$ is not an exceptional group of Lie type (twisted or untwisted)
over $\FF_q$, where $q$ is a prime power not divisible by $p$.
\end{lemma}

\begin{proof}
Assume $G = \L(q)$ is an exceptional group of Lie type over $\FF_q$
with $p \notdivides q$.
Throughout this proof we rely on the Sylow structure for $G$
as described in Proposition~\ref{sylow-structure}.
It shows, in particular, that we
need only consider when the odd prime $p$ divides both order of
the Weyl group of the untwisted group corresponding to $G$ and
$pm_0 \divides m$ for some
$m \in \O(G)$; in all other cases the proposition gives that the Sylow
$p$-subgroup is homocyclic abelian. The cyclotomic factors
$\Phi_m(q)$ and their ``multiplicities'' $r_m$
for each of the exceptional groups are listed
explicitly in \cite[Table~10:2]{GL}.
Note that $3 \divides q^2-1$, so in this
case $m_0$ is 1 or 2; also, $5 \divides q^4 - 1$, so in this case
$m_0$ is 1, 2, or 4; finally, $7 \divides q^6 - 1$, so in this
case $m_0$ is 1, 2, 3, or 6.
In the notation of Proposition~\ref{sylow-structure},
except in the case ${}^3 D_4(q)$ we have
$S = S_T S_W$ (split extension) where $S_T$ is a
normal homocyclic abelian subgroup of
exponent $\ord{\Phi_{m_0}(q)}_p$ and rank $r_{m_0}$, and
$\ord{S_W} = p^b$, where $b$ is defined in
(\ref{weyl-order}).

The exceptional groups are listed in Table~3A along with
$p$ dividing the order of the Weyl group,
permissible $m_0$ such that $m = p^a m_0$ for some $m \in \O(G)$ with $a \ge 1$,
and the corresponding $r_{m_0}$ and $p^b$ for each of these
(in the case of ${}^3 D_4(q)$ we define $3^b$ so that
$\ord S = (\ord{\Phi_{m_0}(q)}_p)^{r_{m_0}} 3^b$).

\begin{figure}[!t]
$$
\vbox{\tabskip=20pt \halign{\hfil #\hfil & \hfil # \hfil & # \hfil \cr
\multispan3 \hfil \bf Table 3A \hfil \cr
\noalign{\bigskip}
Group & Prime $p$ & Permissible $(m_0,r_{m_0},p^b)$   \cr
\noalign{\smallskip}
\noalign{\hrule}
\noalign{\medskip}
${}^3 D_4(q)$ & 3 & \quad $(1,2,3^2)$, $(2,2,3^2)$     \cr
\noalign{\smallskip}
$G_2(q)$     & 3 &  \quad $(1,2,3)$, $(2,2,3)$        \cr
\noalign{\smallskip}
$F_4(q)$     & 3 & \quad $(1,4,3^2)$, $(2,4,3^2)$     \cr
\noalign{\smallskip}
${}^2 F_4(2^n)'$  & 3 & \quad $(2,2,3)$                \cr
\noalign{\smallskip}
$E_6(q)$     & 3 & \quad $(1,6,3^4)$, $(2,4,3^2)$     \cr
             & 5 & \quad $(1,6,5)$                    \cr
\noalign{\smallskip}
${}^2 E_6(q)$ & 3 & \quad $(1,4,3^2)$, $(2,6,3^4)$    \cr
              & 5 & \quad $(2,6,5)$                   \cr
\noalign{\smallskip}
$E_7(q)$  & 3 & \quad $(1,7,3^4)$, $(2,7,3^4)$      \cr
          & 5 & \quad $(1,7,5)$, $(2,7,5)$           \cr
          & 7 & \quad $(1,7,7)$, $(2,7,7)$           \cr
\noalign{\smallskip}
$E_8(q)$  & 3 & \quad $(1,8,3^5)$, $(2,8,3^5)$        \cr
          & 5 & \quad $(1,8,5^2)$, $(2,8,5^2)$, $(4,4,5)$  \cr
          & 7 & \quad $(1,8,7)$, $(2,8,7)$           \cr
}}
$$
\end{figure}

%
%
We consider all these cases, working from largest to
smallest --- the latter requiring more delicate examination.
Table~4-1 in \cite{GL} is used frequently without specific citation: it
lists all the ``large'' subgroups of various families of Lie type
groups that we shall employ. It is helpful to keep in mind the
description of the order of a Sylow $p$-subgroup in
Proposition~\ref{sylow-structure} when comparing the $p$-part of
$\ord G$ to that of its Lie-type subgroups.

\medskip
\noindent
{\bf Case $\pmb{p=7}$}: $E_8(q)$ contains both $A_8(q)$ and ${}^2
A_8(q)$ and so, by inspection of orders, shares a Sylow 7-subgroup
with it in the cases (1,8,7) and (2,8,7) respectively (the Sylow
7-subgroup order is seen to be $7 \cdot \ord{q-\epsilon}^8_7$ for
each group). Likewise $E_7(q)$ contains both $A_7(q)$ and ${}^2
A_7(q)$ and so shares a Sylow 7-subgroup with it in the cases
(1,7,7) and (2,7,7) respectively. By minimality of $G$ all the $p=7$
cases are eliminated.

\medskip
\noindent
{\bf Case $\pmb{p=5}$}: The same containments in the preceding
paragraph for $E_7(q)$ show these groups share a Sylow 5-subgroup
in cases (1,7,5) and (2,7,5).
Similarly, $E_8(q)$ contains $SU_5(q^2)$
and shares a Sylow 5-subgroup with it in the case (4,4,5). By
minimality these $p=5$ cases are eliminated.

Assume $G \iso E_8(q)$. Using the same large subgroups as in the
$p=7$ case, the Sylow 5-subgroup $S$ has a subgroup $S_0$ of index
5 that lies in a subgroup $G_0$ of $G$ of type $A_8(q)$ or ${}^2
A_8(q)$ according to whether we are in cases
$(1,8,5^2)$ or $(2,8,5^2)$ respectively.
By Proposition~\ref{sylow-structure} applied to
$G_0$ it follows that $S_0$ is non-abelian;
and since $\ord A > 5$, $A \cap S_0 \ne 1$.
Thus by induction applied to a minimal
strongly closed subgroup $A_0 \le A \cap S_0$ in $G_0$ we obtain
$S_0 \le A$. Moreover, by Proposition~\ref{sylow-structure} it
follows that $S_T \le S_0$. Since $A$ is non-abelian and since the
normalizer of a Sylow 5-subgroup of the Weyl group of $E_8$ acts
irreducibly on the Sylow 5-subgroup of $W$ (which is abelian of
type (5,5)), the strongly closed subgroup $A$ containing $S_T$
cannot have index 5 in $S$, a contradiction.
This eliminates all $E_8(q)$ cases for $p=5$.

Adopting the notation following Proposition~\ref{sylow-structure},
assume $G \iso E_6^\epsilon(q)$, where $5 \divides q-\epsilon$ and
$S_T$ has rank 6 and index 5 in $S$. Then $G$ shares the Sylow
5-subgroup $S$ with $G_0 = L_1 * L_2$, where $L_1$ and $L_2$ are
central quotients of $SL_2^\epsilon(q)$ and $SL_6^\epsilon(q)$
respectively (both of whose centers have order prime to 5). Since
$A$ is not cyclic, it does not centralize $L_2$; hence it follows
from Lemma~\ref{basic-facts} that $A \cap L_2 \ne 1$. Since $S
\cap L_2$ is non-abelian, by induction $S \cap L_2 \le A$. In
particular, $A$ contains a homocyclic abelian subgroup of rank 5
and exponent $\ord{q-\epsilon}_5$, and $S/A$ is cyclic. Now $G$
also contains a subgroup $G_1 = K_1 * K_2 * K_3$ with each $K_i
\iso SL_3^\epsilon(q)$, where we may assume $S \cap G_1 \in
Syl_5(G_1)$. Each $K_i$ contains a homocyclic abelian subgroup
$B_i$ of rank 2 and exponent $\ord{q-\epsilon}_5$ with
$N_{K_i}(B_i)$ acting irreducibly on $\Omega_1(B_i)$. Because
$S/A$ is cyclic it follows that $B_1 \times B_2 \times B_3 = S_T
\le A$; and since $A$ is non-abelian, $A = S$. This completes the
elimination of all $p=5$ cases.

We next consider the various $p=3$ cases, leaving the nettlesome
groups of type $G_2(q)$ and ${}^3 D_4(q)$ until the very end.

\medskip
\noindent {\bf Case $\pmb{p = 3}$ and $\pmb{m_0 = 1}$}: Here $3
\divides q-1$. If $G \iso F_4(q)$ then it contains the universal
group $G_0 = B_4(q)^u$. By inspection of the order formulas, $G_0$
may be chosen to contain a subgroup $S_0$ of index 3 in $S$ which,
by Proposition~\ref{sylow-structure}, is non-abelian. Since $\ord
A > 3$ we have $S_0 \cap A \ne 1$ so, as usual, the minimality of
$G$ forces $S_0 \le A$. Thus $S_0 = A$ has index 3 in $S$.
Furthermore, since a Sylow 3-subgroup of the Weyl group of $B_4$
has order 3, we get that $A$ has an abelian subgroup of index 3.
But now by \cite[Table~4.7.3A]{GLS3} there is an element $t$ of
order 3 in $G$ such that $C = O^{3'}(C_G(t)) = L_1 * L_2$ where
$L_i \iso SL_3(q)$ for $i=1,2$. Choose a suitable representative
of this class so that $C_S(t) \in Syl_3(C)$. Then $A \cap L_i \nle
Z(L_i)$, so because each Sylow subgroup $S \cap L_i$ is
non-abelian, by induction $S \cap L_i \le A$ for $i=1,2$. This
gives a contradiction because $S \cap L_1 L_2$ clearly does not
have an abelian subgroup of index 3.

Since ${}^2 E_6(q)$ shares a Sylow 3-subgroup with a subgroup
of type $F_4(q)$ this family is eliminated by minimality of $G$.

Consider when $G$ is one of $E_6(q)$, $E_7(q)$ or $E_8(q)$.
In these cases $S_T$ is homocyclic of the same rank as $G$ and
$S_T$ lies in a maximal split torus $T$ of $G$ with $W = N_G(T)/C_G(T)$
isomorphic to the Weyl group of $G$.  Note that $W$ acts on the
Sylow 3-subgroup $S_T$ of $T$; moreover, in each case $W$
acts irreducibly on $\Omega_1(S_T)$, and $Z(S) \le S_T$.
%
%
By strong closure of $A$ we obtain
\begin{equation}
\Omega_1(S_T) \le A. \label{exceptional-1}
\end{equation}
There are containments: $F_4(q) \le E_6(q) \le E_7(q) \le E_8(q)$,
with corresponding containments of their maximal split tori. Thus
by (\ref{exceptional-1}), in each exceptional family $A$
nontrivially intersects a subgroup, $G_0$, of $G$ of smaller rank
in this chain. Since the Sylow 3-subgroups of each $G_0$ are
non-abelian, by minimality of $G$ and the preceding results we
get that $A$ contains a Sylow 3-subgroup of the respective
subgroup $G_0$. Since then $A$ is non-abelian,
it is not contained in $S_T$. Now the Weyl group
of $G$ is of type $U_4(2) \cdot 2$, $Z_2 \times S_6(2)$,
or $2 \cdot O_8^+(2) \cdot 2$, so by induction applied in $N_G(T)$
it follows that $A$ covers a Sylow 3-subgroup of $W$. Finally, the
irreducible action of $W$ on $S_T/\Phi(S_T)$  forces $S_T \le A$,
and so $A = S$, a contradiction.

\medskip
\noindent
{\bf Case $\pmb{p = 3}$ and $\pmb{m_0 = 2}$}: Here $3 \divides q+1$.
The argument employed when $3 \divides q-1$ mutatis mutandis
eliminates $F_4(q)$ as a possibility (using $L_i \iso SU_3(q)$ in
this case).
The groups ${}^2 F_4(2^n)'$ --- including the Tits simple group
--- share a Sylow 3-subgroup with their subgroups $SU_3(2^n)$, and
so are eliminated by induction.
Also, $E_6(q)$ shares a Sylow 3-subgroup with its subgroup $F_4(q)$,
hence it is eliminated.
To eliminate $E_8(q)$, $E_7(q)$ and ${}^2 E_6(q)$ we refer to the
table of centralizers of elements of order 3 in these groups:
\cite[Table 4.7.3A]{GLS3}.

First assume $G \iso E_8(q)$.
By \cite[Table 4.7.3B]{GLS3}, $G$ contains a subgroup
$X \iso L_1 \times L_2$, where the two components are conjugate
and of type $U_5(q)$.
We may assume $S \cap X \in Syl_3(X)$.
Since $\Omega_1(S_T)$ is the unique elementary abelian subgroup
of $S$ of rank 8, $\Omega_1(S_T) \le X$;
in particular, $A \cap X \ne 1$.
As usual, by minimality of $G$ we obtain $S \cap X \le A$,
and the ``toral subgroup'' for $S \cap X$ lies in $S_T$.
Order considerations then give
$S_T \le A$ and $\ind{S}{A} \le 3^3$.
Now the centralizer of an element of order 3 in $Z(S)$
is of type $({}^2 E_6(q) * SU_3(q))3$, where the two
factors share a common center of order 3.
Since $S_T \le A$ it follows that $A$ acts
nontrivially on, hence contains a Sylow 3-subgroup
of, each component (or of $SU_3(2)$ when $q = 2$).
This implies $A$ covers $S/S_T \iso S_W$, as needed to
give the contradiction $A = S$.

Let $G \iso E_7(q)$.  Then $G$ contains a subgroup
$X \iso SU_8(q)$ with $S \cap X \in Syl_3(X)$.
Since $S \cap X$ has the same ``toral subgroup''
as $S$, as usual we obtain $S \cap X \le A$,
$S_T \le A$ and $\ind{S}{A} \le 3^2$.
Now $S$ also contains an element of order 3
whose centralizer has a component of type ${}^2 E_6(q)$
(universal version).
Since as usual $A$ contains a Sylow 3-subgroup of this component
it follows that $A$ covers $S/S_T$ and so $A = S$, a contradiction.

Finally, assume $G \iso {}^2 E_6(q)$. Since by \cite{Atlas} ${}^2
E_6(2)$ shares a Sylow 3-subgroup with a subgroup of type
$Fi_{22}$, by minimality of $G$ we may assume $q > 2$. Let $X$ be
the centralizer of an element of order 3 in $Z(S)$, so $X \iso
(L_1 * L_2 * L_3)(3 \times 3)$, where each $L_i \iso SU_3(q)$, the
central product $L_1 L_2 L_3$ has a center of order 3, an element
of $S$ cycles the three components, and another element of $S$
induces outer diagonal automorphisms on each $L_i$.  As usual, it
follows easily that $A$ contains a Sylow 3-subgroup of $S \cap X$.
By order considerations
$$
\ind{S_T}{S_T \cap A} \le 3
\qquad \text{and} \qquad
\ind{S}{A} \le 9.
$$
Now there is an element $t$ of order 3 in $S$ such that
$$
C = C_G(t) = D * T_1 , \qquad \text{where } D \iso D^-_5(q)
\quad\text{and}\quad T_1 \iso Z_{q+1},
$$
and we may choose $t$ so that $S_0 = C_S(t) \in Syl_3 (C)$.
Let $S_1 = S \cap D$ and $S_2 = S \cap T_1$, and
note that $\gp t = \Omega_1(T_1)$.
Since the Schur multiplier of $D$ has order prime to 3,
$S_0 = S_1 \times S_2$.
It follows as usual that $S_1 \le A$.

Now let $w \in S - S_0$ and let $t_1 = t^w$. Then $t_1 \ne t$ and
$S_0 \in Syl_3(C_G(t_1))$. By symmetry, the strongly closed
subgroup $A$ contains the Sylow 3-subgroup $S_1^w$ of the
component $D^w$ of $C_G(t_1)$. Since $t_1$ acts faithfully on $D$,
so too $S_2^w$ acts faithfully on $D$, from which it follows that
$$
S_2 \le S_1 S_1^w \le A .
$$
Moreover, $A$ contains the ``toral subgroup'' of
$C$ of type $(q+1)^6$ (in the universal version of $G$),
so $S_T \le A$ and hence $A$ is the subgroup of
$S$ that normalizes each component $L_i$ of $X$.
Since $S_W$ is generated by elements of order 3
(in the universal version of $G$),
$S = A\gp x$ for some element $x$ of order 3.
Since no conjugate of $x$ lies in $A$ we may further
assume $C_S(x) \in Syl_3(C_G(x))$.
Since $\gp x$ cycles $L_1,L_2,L_2$ it follows that
the 3-rank of $C_G(x)$ is at most 5:
this restricts the possibilities for the type
of $x$ in \cite[Table 4.7.3A]{GLS3}.
In all possible cases $C_G(x)$ contains a product, $L$,
of one or two components with $C(L)$ cyclic.
The same argument that showed $S_2 \le A$ may now be applied
to show $x \in A$, a contradiction.
This completes the proof for these families.

\medskip
\noindent {\bf Case $\pmb{G_2(q)}$ and $\pmb{ {}^3 D_4(q)}$ where
$\pmb{q \equiv \epsilon \pmod 3}$}: If $G \iso G_2(q)$ then by
Proposition~\ref{normalizer-of-others-in-theorem3} $Z(S) \iso Z_3$
is the unique candidate for $A$, contrary to
Lemma~\ref{A-is-noncyclic}. Thus the minimal counterexample is not
of type $G_2(q)$.

Assume $G \iso {}^3 D_4(q)$.  Then $G$ contains a subgroup $G_0$
isomorphic to $G_2(q)$ (the fixed points of a graph automorphism
of order 3), and by order considerations we may assume
$S_0 = S \cap G_0$ is Sylow in $G_0$ and so has index 3 in $S$.
As noted above, $\gp z = Z(S_0)$ is of order 3
and is the unique nontrivial strongly closed (in $G_0$)
proper subgroup of $S_0$.
Consider first when $\ord{A \cap S_0} > 3$. Then since $S_0$ is
non-abelian, induction applied to $G_0$
gives $S_0 \le A$, and so $A = S_0$. Since by
Proposition~\ref{proposition3-4},
$z^{G_0} \cap S_0 = \{ z^{\pm 1} \}$,
whereas $\gp z$ is not strongly closed in $G$, there must be
$G$-conjugates of  $z$ in $S - S_0$, contrary to $A$ being
strongly closed (one can see this fusion in a subgroup of
${}^3 D_4(q)$ of type $PGL^\epsilon_3(q)$).

Thus $A \cap S_0 =\gp z$ and so by Lemma~\ref{A-is-noncyclic},
$A = \gp z \times \gp y$ with $z \sim y$ in $G$.
Since $[S,y] \le \gp z$, $y$ centralizes $\Phi (S)$.
Since ${}^3 D_4(q)$ has 3-rank 2
and $y \notin \Phi(S)$, by Proposition~\ref{sylow-structure}(4) we
must have $\ord S = 3^4$.  But then $S_0$ is the non-abelian group
of order 27 and exponent 3, and $y$ centralizes a subgroup of
index 3 in it, contrary to the 3-rank of ${}^3 D_4(q)$ being 2.
This eliminates the possibility that $G \iso {}^3 D_4(q)$ and so
completes the consideration of all cases.
\end{proof}


\begin{lemma}
\label{G-not-char-p}
$G$ is not a group of Lie type (untwisted or twisted) in
characteristic $p$.
\end{lemma}

\begin{proof}
Assume $G$ is of Lie type (untwisted or twisted) over $\FF_q$
where $q = p^n$. Since $G$ is a counterexample, it follows from
Proposition~\ref{irreducible-action} that $G$ has $BN$-rank $\ge
2$. An end-node maximal parabolic subgroup $P_1$ for each of the
Chevalley groups (untwisted or twisted) containing the Borel
subgroup $S$ is described in detail in \cite{CKS} and
\cite{GLS-BN} (for the classical groups these parabolics are the
stabilizers in $G$ of a totally isotropic one-dimensional subspace
of the natural module.)  For the groups of $BN$-rank 2 the other
maximal parabolic, $P_2$, is also described in \cite{GLS-BN}. In
each group $P_i = Q_i L_i H$, where $Q_i = O_p (P_i)$, $L_i$ is
the component of a Levi factor of $P_i$ and $H$ is a $p'$-order
Cartan subgroup.

Except for the 5-dimensional unitary groups and some groups over
$\FF_3$ (which will be dealt with separately), for some $i \in
\{1,2\}$ the group $M = O^{p'}(P_i)$ satisfies the following
conditions:

\medskip
\begin{paragraph}
{\bf Properties~\ref{charp-hyp}A}

\begin{enumerate}

\item[(1)]
$S \le M$,

\item[(2)]
$F^*(M) = O_p(M)$,

\item[(3)]
$\overline M = M/O_p(M)$ is a quasisimple group of Lie type in
characteristic $p$,

\item[(4)]
$\overline M$ is not isomorphic to $U_3(p^n)$ or $Re(3^n)$ (when
$p=3$), for any $n \ge 2$,

\item[(5)]
$[O_p(M),\overline M] = O_p(M)$, and

\item[(6)]
if $Q = O_p(M)$ and $Z = \Omega_1(Z(S))$, then one of the
following holds:

\begin{description}

\item[(i)]
$Q$ is elementary abelian of order $q^k$ for some $k$, or

\item[(ii)]
$Q$ is special of type $q^{1{+}k}$ for some $k$, all subgroups of
order $p$ in $Z$ are conjugate in $G$, and $z^g \in S - Q$ for
some $z \in Z$, $g \in G$.

\end{description}
\end{enumerate}
\label{charp-hyp}
\end{paragraph}

\medskip

Basic information about this parabolic is listed in Table 3B.  The
last column of Table~3B indicates which of the two alternatives in
Properties~\ref{charp-hyp}A(6) holds. The proofs that the fusion in
Properties~\ref{charp-hyp}A(6ii) holds in each case may be found in
\cite{CKS}.

$$
\vbox{ \tabskip=20pt \halign{#\hfil &\hfil # \hfil &\hfil # \hfil
&\hfil # \hfil &\hfil # \hfil \cr
\multispan5 \hfil \bf Table 3B \hfil \cr
\noalign{\bigskip}
Group  & Parabolic & $Q$ & $L/Z(L)$ & \ref{charp-hyp}A(6) \cr
\noalign{\smallskip}
\noalign{\hrule}
\noalign{\medskip}
$L_k (q)$, $k \ge 3$     & $P_1$  & $q^{k{-}1}$
& $L_{k{-}1} (q)$   & (i) \cr
$O_{k}^\pm (q)$, $k \ge 7$ & $P_1$  & $q^{k{-}2}$
    & $O_{k{-}2}^\pm (q)$ & (i)   \cr
$S_{2k} (q)$, $k \ge 2$  & $P_1$  & $q^{1+ 2(k{-}1)}$ &
$S_{2k{-}2}(q)$   & (ii) \cr
$U_k(q)$, $k \ge 4$, $k \ne 5$   & $P_1$  & $q^{1+2(k{-}2)}$
    & $U_{k{-}2}(q)$    & (ii) \cr
$E_6(q)$        & $P_1$     & $q^{1+20}$    & $L_6(q)$          &
(ii) \cr
$E_7(q)$        & $P_1$     & $q^{1+32}$    & $O_{12}^+ (q)$    &
(ii) \cr
$E_8(q)$        & $P_1$     & $q^{1+56}$    & $E_7(q)$          &
(ii) \cr
${}^2 E_6(q)$   & $P_1$     & $q^{1+20}$    & $U_6 (q)$         &
(ii) \cr
$G_2(q)$, $q > 3$       & $P_2$     & $q^{1+4}$     & $L_2(q)$
& (ii) \cr
$F_4(q)$        & $P_1$     & $q^{1+14}$    & $S_6 (q)$         &
(ii) \cr
${}^3 D_4(q)$   & $P_2$     & $q^{1+8}$     & $L_2(q^3)$        &
(ii) \cr
\noalign{\medskip}
$U_5(q)$        & $P_1$     & $q^{1+6}$       & $U_3(q)$
& (ii)  \cr
}}
$$

\medskip

Putting aside the last row for the moment, let $M = O^{p'}(P_i)$
be chosen according to Table 3B. Since $M$ does not have any
composition factors isomorphic to $U_3(p^n)$ or $Re(3^n)$, the
minimality of $G$ gives inductively that $A \in Syl_p(\gp{A^M})$.
If $A \nle Q$, then by the structure of $M$ in
Properties~\ref{charp-hyp}A(3) and (5), $M \le \gp{A^M}$.  But then
$A = S$ by (1), a contradiction. Thus
\begin{equation}
A \le Q \quad \text{and} \quad  A \nor M . \label{charp-1}
\end{equation}

Assume first that Properties~\ref{charp-hyp}A(6ii) holds. Then
since $A \nor S$, $Z \cap A \ne 1$.  The strong closure of $A$
together with (6ii) forces $Z \le A$, contrary to the existence of
some $z^g \in S - Q$.  This contradiction shows that $G$ can only
be among the families in the first two rows or the last row of
Table~3B.

Assume now that $Q$ is abelian, i.e., $G$ is a linear or
orthogonal group. In these cases $Q$ is elementary abelian and is
the natural module for $\overline M$; in particular, $\overline M$
acts irreducibly on $Q$.  By (\ref{charp-1}) we obtain $A = Q$.
However, in these cases when $G$ is viewed as acting on its
natural module, $Q$ is a subgroup of $G$ that stabilizes the
one-dimensional subspace generated by an isotropic vector and acts
trivially on the quotient space. Since the dimension of the space
is at least 3, one easily exhibits noncommuting transvections that
stabilize a common maximal flag; hence there are conjugates of
elements of $Q$ in $S$ that lie outside of $Q$, a contradiction.

In $U_5(q)$ for $q \ge 3$ the unipotent radical of the parabolic
$P_1$ is special of type $q^{1+6}$ with $Z = Z(S) = Z(Q_1)$ and
all subgroups of order $p$ in $Z$ conjugate in $P_1$ (so $Z \le
A$).  As in the other unitary groups, there exist $z \in Z$ and $g
\in G$ such that $z^g \in S - Q_1$. Now $L_1 \iso U_3(q)$ acts
irreducibly on $Q_1 /Z$ and, by the strong closure of $A$, $A \cap
Q$ is normal in $P_1$. Since $z^g \in A$ and $[Q_1,z^g] \le A \cap
Q_1$, the irreducible action of $L_1$ forces $Q_1 \le A$.  But now
there is a root group $U$ of type $U_3(q)$ with $U$ contained in
$Q_1$ such that $S = Q_1 U^x$, for some $x \in G$.  Since $U \le
A$, this forces $A = S$, a contradiction.

It remains to treat the special cases when the Levi factors in
Table~3B are not quasisimple: $G \iso L_2(q)$, $L_3(3)$, $G_2(3)$,
$S_4(3)$, or $U_4(q)$ (in line~3 of Table~3B, $S_2(q) = L_2(q)$).
Properties of small order groups may be found in \cite{Atlas}. The
groups $L_2(q)$ have elementary abelian Sylow $p$-subgroups so $G$
is not a counterexample in this instance. In $L_3(3)$ we
have $S \iso 3^{1+2}$ and the action of the two maximal parabolic
subgroups (stabilizers of one- and two-dimensional subspaces)
easily show that the strong closure $Z(S)$ in $S$ is all of $S$,
contrary to $A \ne S$.

If $G \iso G_2(3)$ then since $G$ has two (isomorphic) maximal
parabolics containing $S$, $A$ is not normal in one of them, say
$P_1$. By \cite{Atlas}, $P_1 = (W \times U): L$ where $W \iso
3^{1+2}$, $U \iso Z_3 \times Z_3$, $O_3(P_1) = WU$, and $L \iso
GL_2(3)$ acts naturally on both $U$ and $W/W'$. Since $A$ projects
onto a subgroup of order 3 in $P_1/O_3(P_1) \iso L$, we see that
$[A,W] \nle W'$ and $[A,U] \ne 1$.  Both these commutators lie in
the strongly closed subgroup $A$, so the action of $L$ forces
$O_3(P) \le A$.  Thus $A = S$, a contradiction.

If $G \iso S_4(3)$ there are maximal parabolics of type $P_1 =
3^{1+2} : SL_2(3)$ and $P_2 = 3^3 (S_4 \times Z_2)$. Since $P_1 =
N_G(Z(S))$ it follows that the $S_4$ Levi factor in $P_2$ acts
irreducibly on $O_3(P_2)$. Now $A \cap O_3(P_2) \ne 1$ so
$O_3(P_2) \le A$. Likewise since $A$ is a noncyclic strongly
closed subgroup, it follows easily from the action of the Levi
factor in $P_1$ that $O_3(P_1) \le A$.  These together give $A =
S$, a contradiction.

Finally, assume $G \iso U_4(q)$.  From the isomorphism $U_4(q)
\iso O_6^+(q)$ we see that $G$ contains a maximal parabolic $P_2 =
q^4 O_4^+(q) \iso q^4 L_2(q^2)$, where the Levi factor is
irreducible on the (elementary abelian) unipotent radical.  This
case has been eliminated by previous considerations. This final
contradiction completes the proof of the lemma.
\end{proof}

\begin{lemma}
\label{G-not-sporadic}
$G$ is not one of the sporadic simple groups.
\end{lemma}

\begin{proof}
The requisite properties of the sporadic groups for this proof are
nicely documented in \cite{Atlas}, \cite[Section 5]{GL}, or
\cite[Section 5.3]{GLS3}; many of their proofs may be found in
\cite{AS}. Facts from these sources are quoted without further
attribution. Verification that the sporadic groups in conclusions
(iv) and (v) of Theorem~\ref{theorem3-1} indeed have strongly
closed subgroups as asserted may also be found in these
references. We clearly only need to consider groups where $p^2$
divides the order; indeed, when the Sylow $p$-subgroup has order
exactly $p^2$ it is elementary abelian and $G$ is not a
counterexample in these cases.

If $\ord S = p^3$, then in all cases the Sylow $p$-subgroup is
non-abelian of exponent $p$ and, with the exception of $M_{12}$,
$N_G(S)$ acts irreducibly on $S/Z(S)$. In $M_{12}$ with $p = 3$:
$S$ contains distinct subgroups $U_1$ and $U_2$, each of order 9,
such that $N_G(U_i)$ acts irreducibly on $U_i$ for each $i$. Since
$A$ is noncyclic and strongly closed, in all cases these
conditions force $A = S$, a contradiction. Thus we are reduced to
considering when $\ord S \ge p^4$.

We first argue that the following general configuration cannot
occur in $G$:

\medskip
\begin{paragraph}
{\bf Properties~\ref{sporadic-hyp}B}

\begin{enumerate}

\item[(1)]
$Z(S) = Z \iso Z_p$,

\item[(2)]
$N = N_G(Z)$ has $Q = O_p(N)$ extraspecial of exponent $p$ and
width $w > 1$ (denoted $Q \iso p^{1{+}2w}$),

\item[(3)]
$N$ acts irreducibly on $Q/Q'$, and

\item[(4)]
$N/Q$ does not have a nontrivial strongly closed $p$-subgroup that
is proper in a Sylow $p$-subgroup of $N/Q$.

\end{enumerate}
\label{sporadic-hyp}
\end{paragraph}
\medskip

By way of contradiction assume these conditions are satisfied in
$G$. If $A \nle Q$ then by (4) we obtain that $A$ covers a Sylow
$p$-subgroup of $N/Q$. In this case, the irreducible action of $N$
on $Q/Q'$ then forces $Q \le A$ and so $A = S$, a contradiction.
Thus $A \le Q$. Now $Z \le A$ but $\ord A > p$ so the irreducible
action of $N$ forces $A = Q$. Since $A$ is minimal strongly
closed, whence $Z$ is not strongly closed, there is some $x \in Q
- Z$ such that $x \sim z$ for $z \in Z$. Thus by Sylow's Theorem
there is some $g \in G$ such that
$$
C_Q(x)^g \le S \qquad \text{and} \qquad x^g = z.
$$
By strong closure, $C_Q(x)^g \le Q$.  But since $Q$ has width $>
1$ we obtain $Z^g \le (C_Q(x)^g)' \le Q' = Z$ and so $g$
normalizes $Z$.  This contradicts the fact that $z^{g^{-1}} \notin
Z$ and so proves these properties cannot hold in $G$.

Most sporadic groups are eliminated because they satisfy
Properties~\ref{sporadic-hyp}B, or because they share a Sylow
$p$-subgroup with a group that is eliminated inductively. All
cases where $\ord S \ge p^4$ are listed in Table~3C along with the
isomorphism type of the corresponding normalizer of a $p$-central
subgroup (or another ``large'' subgroup, or reason for
elimination). Some additional arguments must be made in a few
cases.

\begin{figure}[!t]
$$
\vbox{ \tabskip=20pt \halign{#\hfil & # \hfil  \cr \multispan2
\hfil \bf Table 3C \hfil \cr \noalign{\medskip} \bf Group  &   \bf
$\bold {Z(S)}$ normalizer (or other reason) \cr
\noalign{\smallskip}
\noalign{\hrule}
\noalign{\medskip}
\multispan 2 \hfil $\pmb{p = 7}$ \hfil \cr
\noalign{\medskip}
$M$ & $7^{1{+}4} (3 \times 2 S_7)$   \cr
\noalign{\medskip}
\multispan 2 \hfil $\pmb{p = 5}$ \hfil                 \cr
\noalign{\medskip}
$Ly$     & $5^{1{+}4} ((4*SL_2(9)). 2)$      \cr
$Co_1$   & $5^{1{+}2} GL_2(5)$                 \cr
$HN$     & $5^{1{+}4}(2^{1{+}4}(5 \cdot 4))$         \cr
$B$      & $5^{1{+}4}(((Q_8*D_8)A_5)\cdot 4)$  \cr
$M$      & $5^{1{+}6}((4*2J_2)\cdot 2)$        \cr
\noalign{\medskip}
\multispan 2 \hfil $\pmb{p = 3}$ \hfil                 \cr
\noalign{\medskip}
$McL$         & $3^{1{+}4}(2S_5)$           \cr
$Suz$         & $3 U_4(3)2$                      \cr
$Ly$          & $3 McL 2$                        \cr
$O\text{'}N$  & (one class of $Z_3$ and $S = \Omega_1(S))$           \cr
$Co_1$        & $3^{1{+}4}GSp_4(3)$              \cr
$Co_2$        & $3^{1{+}4}((D_8*Q_8)\cdot S_5)$  \cr
$Co_3$        & $3^{1{+}4}((4*SL_2(9))\cdot 2)$  \cr
$Fi_{22}$     & ($S \le O_7(3)$)                 \cr
$Fi_{23}$     & ($S \le O_8^+(3):S_3$)           \cr
$Fi_{24}'$    & $3^{1{+}10}(U_5(2)\cdot 2)$      \cr
$HN$          & $3^{1{+}4} (4*SL_2(5))$          \cr
$Th$          & (see separate argument)          \cr
$B$           & $3^{1{+}8} (2^{1{+}6} O_6^-(2))$ \cr
$M$           & $3^{1{+}12} (2 Suz)\cdot 2$      \cr
}}
$$
\end{figure}

\medskip
When $p = 5$ and $G \iso Co_1$ the extraspecial $Q = O_5(N)$
listed in the table has width~1. As before, if $A \nle Q$ then the
irreducible action of $N$ on $Q/Q'$ forces $A = S$, a
contradiction. Thus $A \le Q$ and again the irreducible action
yields $A = Q$. However $G$ contains a subgroup $G_0 \iso Co_2$
whose Sylow 5-subgroup $S_0$ is isomorphic to $Q$ and has index 5
in $S$. Since $\ord{A \cap S_0} \ge 25$, the irreducible action of
$N_{G_0}(S_0)$ on $S_0/S_0'$ forces $S_0 \le A$, and hence $S_0 =
A$. But by Proposition~\ref{proposition3-4}, $Z(S_0)$ is strongly
closed in $G_0$ but not strongly closed in $G$.  Thus there is
some $g \in G$ such that $Z(S_0)^g \le S$ but $Z(S_0)^g \nle S_0$.
This contradicts the fact that $A = S_0$ is strongly closed in
$G$, and so $G \not\iso Co_1$.

When $p = 3$ and $G \iso Fi_{23}$ it contains a subgroup $H$ of
type $O_8^+(3):S_3$ that may therefore be chosen to contain $S$.
Let $H_0 = H'' \iso O_8^+(3)$.
By Lemma~\ref{basic-facts}, $A \cap H_0 \ne 1$;
and so by induction $A$ contains the non-abelian Sylow
3-subgroup $S_0 = S \cap H_0$ of $H_0$. Thus $\ind{S}{A} = 3$. Now
$H$ is generated by 3-transpositions in $G$, and so there are
3-transpositions $t,t_1$ such that
$$
D_1 = \gp{t,t_1} \iso S_3 \qquad \text{and}\qquad H = H_0 : D_1.
$$
Likewise $t$ inverts some element of order 3 in $H_0$, i.e., there
is some $t_2 \in H_0 \gp t$ such that $D_2 = \gp{t,t_2} \iso S_3$.
By the rank 3 action of $G$ on its 3-transpositions, $D_1$ and
$D_2$ are conjugate in $G$. Thus $D_1'$ is conjugate to the
subgroup $D_2'$ of $H_0$, contrary to $A$ being strongly closed.
This proves $G \not\iso Fi_{23}$.

Finally, assume $p = 3$ and $G \iso Th$. Following the Atlas
notation and the computations in \cite{W}, the centralizer of an
element of type $3A$ in $S$ has isomorphism type
$$
N = N_G(\gp{3A}) \iso (Z_3 \times H).2 \qquad \text{where} \qquad
H \iso G_2(3).
$$
Since an element of type $3B$ in $Z(S) \cap A$ commutes with $3A$
and therefore acts nontrivially on $H$, by induction $A$ contains
a Sylow 3-subgroup of $H$. In the Atlas notation for characters of
$G_2(3)$, the character $\chi$ of degree 248 of $Th$ restricts to
$Z_3 \times H$ as
$$
\chi|_{Z_3 \times H} = 1 \tensor (\chi_1 + \chi_6) + (\omega +
\overline\omega) \tensor \chi_5
$$
where the characters of the $Z_3$ factor are denoted by their values
on a generator.  By comparison of the values of these on the
$G_2(3)$-classes it follows that $H$ contains a representative of
every class of elements of order 3 in $Th$.
The calculations in \cite{W} show that $S = \Omega_1(S)$,
which leads to $A = S$, a contradiction.

This eliminates all sporadic simple groups as potential counterexamples,
and so completes the proof of Theorem~\ref{theorem3-2}.

\end{proof}

\subsection{The Proof of Theorem~\ref{theorem3-1}}
\label{theorem3-1-proof}
%
\    
%


This subsection derives Theorem~\ref{theorem3-1} as a consequence
of Theorem~\ref{theorem3-2}.
Throughout this subsection $G$ is a minimal counterexample to
Theorem~\ref{theorem3-1}.

Since strong closure inherits to quotient groups,
if $\O_A(G) \ne 1$ we may apply induction to $G/\O_A(G)$
and see that the asserted conclusion holds.
Thus we may assume $\O_A(G) = 1$, and consequently
\begin{equation}
\gathered
\text{$A \cap N$ is not a Sylow $p$-subgroup of $N$ for any
nontrivial $N \nor G$} \\
\text{and $O_{p'}(G) = 1$.}
\endgathered
\label{theorem3.1-proof-1}
\end{equation}
Likewise if $G_0 = \gp{A^G}$ then by Frattini's Argument, $G = G_0
N_G(A)$, whence $\gp{A^G} = \gp{A^{G_0}}$.  Thus we may replace
$G$ by $G_0$ to obtain
\begin{equation}
\text{$G$ is generated by the conjugates of $A$.}
\label{theorem3.1-proof-2}
\end{equation}
By strong closure $A \cap O_p(G) \nor G$, whence
by (\ref{theorem3.1-proof-1}), $A \cap O_p(G) = 1$.
Since $[A,O_p(G)] \le A \cap O_p(G) = 1$, by
(\ref{theorem3.1-proof-2}) we have
\begin{equation}
O_p(G) \le Z(G). \label{theorem3.1-proof-3}
\end{equation}
Consequently $F^*(G)$ is a product of subnormal quasisimple
components $L_1,\dots,L_r$ with $O_{p'}(L_i) = 1$ for all $i$.
Moreover $S_i = S \cap L_i$ is a Sylow $p$-subgroup of $L_i$ and
$S_i \ne 1$ by (\ref{theorem3.1-proof-1}).

We argue that each component of $G$ is normal in $G$. By way of
contradiction assume $\{L_1, \dots, L_s\}$ is an orbit of size
$\ge 2$ for the action of $G$ on its components. Let $Z = A \cap
Z(S)$, so that $Z$ normalizes each $L_i$. Thus $N = \cap_{i=1}^s
N_G(L_i)$ is a proper normal subgroup of $G$ possessing a
nontrivial strongly closed $p$-subgroup, $B = A \cap N$ that is
not a Sylow subgroup of $N$. By induction --- keeping in mind that
components of $N$ are necessarily components of $G$ and $\O_B(N) =1$
--- and after possible renumbering, there are simple components
$L_1,\dots,L_t$ of $N$ that satisfy the conclusion of
Theorem~\ref{theorem3-1} with $B \cap L_i \ne 1$, these are all
the components of $N$ satisfying the latter condition, and $t \ge
1$. By Frattini's Argument $G = N_G(B) N$ from which it follows
that $L_1 \cdots L_t \nor G$. The transitive action of $G$ in turn
forces $t=s$. Thus $A$ permutes $\{L_1,\dots,L_s\}$ and $1 \ne A
\cap L_i < S_i$. If $A$ does not normalize one of these
components, say $L_i^a = L_j$ for some $i \ne j$ and $a \in A$,
then $S_i S_i^a = S_i \times S_j$. But then $[S_i,a] \nle (A \cap
L_i) \times (A \cap L_j)$, contrary to $A \nor S$. Thus $A$ must
normalize $L_i$ for $1 \le i \le s$. Since $A \le N \nor G$,
(\ref{theorem3.1-proof-2}) gives $N = G$, a contradiction.  This
proves
\begin{equation}
\text{every component of $G$ is normal in $G$.}
\label{theorem3.1-proof-4}
\end{equation}
The preceding results also show that $A$ acts nontrivially on each
$L_i$. By Lemma~\ref{basic-facts}, $A_i = A \cap L_i \ne 1$ and
$A_i$ is not Sylow in $L_i$ for every $i$. By
Theorem~\ref{theorem3-2} applied to each $L_i$ using a minimal
strongly closed subgroup of $A_i$ we obtain
\begin{equation}
\text{$F^*(G) = L_1 \times L_2 \times \cdots \times L_r$}
\label{theorem3.1-proof-5}
\end{equation}
and each $L_i$ is one of the simple groups
described in the conclusion of Theorem~\ref{theorem3-1}.
Moreover, in each of conclusions (i) to (v), by
Propositions~\ref{irreducible-action}
and \ref{normalizer-of-others-in-theorem3},
$A_i$ is a subgroup of $L_i$ described in the respective conclusion.

It remains to verify that the action of $A$ is as claimed when $A \nle F^*(G)$.
The automorphism group of each $L_i$ is described in
detail in \cite[Theorem 2.5.12 and Section 5.3]{GLS3} --- these
results are used without further citation.

Let $S^* = S \cap F^*(G) = S_1 \times \cdots \times S_r$,
let $H^* = H_1 \times \cdots \times H_r$,
where $H_i$ is a $p'$-Hall complement to $S_i$ in $N_{L_i}(S_i)$,
and let $N^* = A S^*H^*$.
Note that $O_{p'}(N^*) = C_{H^*}(S^*)$ is $A$-invariant.
Now in all cases $[A,S_i] \le A \cap S_i \le \Phi(S_i)$, that is,
$A$ commutes with the action of $H^*$ on $S^*/\Phi(S^*)$.
This forces $A \le O_{p',p}(N^*)$.
By strong closure of $A$ we get that $AO_{p'}(N^*) \nor N^*$.
Thus $N_{N^*}(A)$ covers $H^*/C_{H^*}(S^*)$.
Let $H$ be a $p'$-Hall complement to $AS^*$ in $N_{N^*}(A)$;
we may assume $H^* = HC_{H^*}(S^*)$.
We have a Fitting decomposition
\begin{equation}
A = [A,H] A_F \quad \text{where} \quad A_F = C_A(H).
\label{theorem3.1-proof-6}
\end{equation}
By Propositions~\ref{irreducible-action}
and \ref{normalizer-of-others-in-theorem3} each $A_i$ is abelian
and $H_i$, hence also $H$, acts without fixed points on each $A_i$.
Since $[A,H] \le A \cap F^*(G)$ we therefore
obtain
\begin{equation}
[A,H] = A_1 \times \cdots \times A_r \qquad \text{and} \qquad A_F
\cap [A,H] = 1. \label{theorem3.1-proof-7}
\end{equation}
We now determine the action of $A_F$ on $L_i$ for each isomorphism
type in conclusions (i) to (v).

First suppose $A_F$ acts trivially on some $L_i$, say for $i=1$.
In this situation $A = A_1 \times B$ where
$B = (A_2 \times \cdots \times A_r)A_F = A \cap C_G(L_1)$.
Then $\gp{A^G} = L_1 \times \gp{B^G}$, and so we may proceed inductively
to identify $\gp{B^G}$ and conclude that Theorem~\ref{theorem3-1} is valid.
We now observe that $A_F$ acts trivially on all components listed
in conclusions (ii) to (v) as follows:
If $L_1$ is one of these cases, it follows from
Proposition~\ref{normalizer-of-others-in-theorem3}
that $C_{H_1}(S_1) = 1$ and so $A_F$ centralizes a
$p'$-Hall subgroup of $N_{L_1}(S_1)$.
In case (ii) of the conclusions, if $L_1$ is a Lie-type simple
group in characteristic $p$ and $BN$-rank 1,
by \cite[9-1]{GL} no automorphism of
order $p$ centralizes a Cartan subgroup of $L_1$,
so $A_F$ acts trivially on $L_1$.
If $L_1 \iso G_2(q)$ is described by case (iii) of the conclusion,
then since $[S_1,A_F] \le A_1$, the last assertion of
Proposition~\ref{normalizer-of-others-in-theorem3}(3)
shows that $A_F$ acts trivially on $L_1$.
And in cases (iv) and (v) of the conclusions,
when $L_1$ is a sporadic group,
none of the target groups admits an outer automorphism
of order $p$, and no inner automorphism that normalizes a Sylow
$p$-subgroup also commutes with a $p'$-Hall subgroup of its
normalizer.  Thus $A_F$ acts trivially in these instances too.

It remains to consider when every $L_i$ is described by conclusion (i):
$L = L_i$ is a group of Lie type over the field
$\FF_{q_i}$ where $p \notdivides q_i$ and the Sylow $p$-subgroups are
abelian but not elementary abelian.
Since $A_F$ commutes with the action
of a $p'$-Hall subgroup of $N_L(S_i)$,
it follows from Proposition~\ref{irreducible-action} that
$A_F$ induces outer automorphisms on $L$.
The outer diagonal automorphism group of $L$ has order
dividing the order of the Schur multiplier of $L$, so
by Proposition~\ref{sylow-structure}(7) no element of $G$
induces a nontrivial outer diagonal automorphism of
$p$-power order on $L$.
Since Sylow 3-subgroups of $D_4(q)$ and ${}^3 D_4(q)$ are
non-abelian, $L$ does not admit a nontrivial
graph or graph-field automorphism when $p=3$.
This shows $A_F$ must act as field automorphisms on $L$,
and hence $A_F/C_{A_F}(L)$ is cyclic.

Now $G$ is generated by the conjugates of $A$, hence
the group $\widetilde G = G/LC_G(L)$ of outer automorphisms on $L$
is generated by conjugates of $\widetilde A_F$.
This implies via \cite[Theorem 2.5.12]{GLS3} that
\begin{equation}
\label{diagonal-autos}
\widetilde G = \widetilde D \widetilde A_F
\qquad \text{and} \qquad
\widetilde D = [\widetilde D,\widetilde A_F]
\end{equation}
where $\widetilde D$ is a cyclic $p'$-subgroup of
the outer-diagonal automorphism group of $L$
normalized by the cyclic $p$-group $\widetilde A_F$ of
field automorphisms.
Moreover, since $p > 3$ when $L$ is of type $E_6(q)$,
${}^2 E_6(q)$ or $D_{2m}(q)$, the action of
$\widetilde A_F$ on $\widetilde D$ in (\ref{diagonal-autos})
implies that $\widetilde D$ is trivial except
in the cases where $L$ is a linear or unitary group.

A $p'$-order subgroup $D$ that covers the section
$\widetilde D$ for every $L_i$ may be defined as follows
(even in the presence of $L_i$ that are not of type (i)):
We have now established that $S = S^*A_F$,
and that $S^*$ is a Sylow $p$-subgroup of the (normal)
subgroup $G_D$ of $G$ inducing only inner and diagonal automorphisms
on $F^*(G)$.  Thus $N_{G_D}(S^*)$ has a $p'$-Hall complement, which
is then a complement to $S = S^*A_F$ in $N_G(S^*)$.
Since $[S^*,A_F] \le \Phi(S^*)$,
$A_F$ commutes with the action on $S^*$ of this
$p'$-Hall subgroup.
As $\widetilde D = [\widetilde D,A_F]$,
any choice for $D$ must lie in $C_G(S^*)$.
However, $C_G(S^*)$ has a normal $p$-complement, so
any $D$ must lie in $O_{p'}(C_G(S^*))$.
Thus $[O_{p'}(C_G(S^*)),A_F] = [O_{p'}(C_G(S^*)),S]$
covers $\widetilde D$ for every component $L_i$
(and centralizes all components that are not of type
$PSL$ or $PSU$).

Finally note that in every case $A_F'$ centralizes $L_i$ for every $i$.
Since then $A_F'$ centralizes $F^*(G)$, it must be trivial,
that is, $A_F$ is abelian. Since $A_F/C_{A_F}(L_i)$ is cyclic for
all $i$, it follows that $A_F = A_F/\cap_{i=1}^r C_{A_F}(L_i)$ has rank
at most $r$, as asserted. This completes the proof of
Theorem~\ref{theorem3-1}.
\hfill \framebox(7,7){}\break


\subsection{The Proofs of Corollaries~\ref{corollary3-3} and \ref{corollary3-4}}
\label{corollary3-3-proof}
%
\    
%




Considering both corollaries at once,
assume the hypotheses of Corollary~\ref{corollary3-3} hold. The
result is trivial if either $A = S$ (in which case $\O_A(G) = G$)
or $A = 1$ (in which case $G = 1$). By passing to $G/\O_A(G)$ we
may assume $\O_A(G) = 1$. Since $G$ is generated by conjugates of
$A$, Theorems~\ref{theorem3-1-p-equal-2} and \ref{theorem3-1}
imply that
\begin{equation}
G = (L_1 \times \cdots \times L_r)(D \cdot A_F),
\label{corollary3-1-proof-1}
\end{equation}
where the $L_i$, $D$ and $A_F$ are described in their conclusions
(with both $D$ and $A_F$ trivial when $p = 2$).
Let $S_i = S \cap L_i$ and $A_i = A \cap L_i$.

For each $i$ let $Z_i$ be a minimal nontrivial
strongly closed subgroup of $A \cap L_i$, and let
$Z = Z_1 \times \cdots \times Z_r$.
Then $Z$ is strongly closed in $G$, and
by Propositions~\ref{irreducible-action} and
\ref{normalizer-of-others-in-theorem3},
$Z$ is contained in the center of $S$.
It is immediate from Sylow's Theorem and the weak closure
of $Z$ that $N_G(Z)$ controls strong $G$-fusion in $S$.
Now
$$
N_G(Z) = (N_{L_1}(Z_1) \times \cdots \times N_{L_r}(Z_r))(D \cdot A_F)
$$
where by the proof of Theorem~\ref{theorem3-1},
$D = [D,A_F]$ may be chosen to be an $S$-invariant
$p'$-subgroup centralizing each $S_i$.
This implies
\begin{equation}
\label{control-fusion}
\text{$M = (N_{L_1}(Z_1) \times \cdots \times N_{L_r}(Z_r)) A_F$
\ controls strong $G$-fusion in $S$.}
\end{equation}
It suffices therefore to show that $N_M(A)$ controls strong
$M$-fusion in $S$.
Furthermore, $N_M(A)$ controls strong $M$-fusion in $S$ if and
only if the corresponding fact holds in $M/O_{p'}(M)$;
so we may pass to this quotient and therefore assume $O_{p'}(M) = 1$
(without encumbering the proof with overbar notation, since all normalizers
considered are for $p$-groups).

If $L_i$ is a Lie-type component with $S_i$ abelian then, as noted in the
proof of Theorem~\ref{theorem3-1}, $N_{L_i}(Z_i) = N_{L_i}(S_i)$ and
$A_F$ commutes with the action on $S_i$ of an $A_F$-stable $p'$-Hall subgroup
$H_i$ of this normalizer.  Since $O_{p'}(M) = 1$ it follows that
$H_i$ acts faithfully on $S_i$, and so $[A_F,H_i] = 1$.
On the other hand, if $L_i$ is not of this type, $[A_F,L_i] = 1$.
Thus (reading modulo $O_{p'}(M)$) we have
\begin{equation}
\label{field-action}
M = SC_M(A_F)
\end{equation}
and so $N_M(A) = N_M(A^*)$, where $A^* = A_1 \cdots A_r$.

For every component $L_i$ that is not of type $G_2(q)$ or $J_2$,
by Corollary~\ref{normalizer-of-A-all},
$N_{L_i}(Z_i) = N_{L_i}(S_i)$; and therefore in these components
$N_{L_i}(Z_i) = N_{L_i}(A_i)$ too.
However, for a component $L_i$ of type $G_2(q)$ or $J_2$
(with $p=3$), by Proposition~\ref{normalizer-of-others-in-theorem3}
we must have $Z_i = A_i$.
In all cases we have $N_{L_i}(Z_i) = N_{L_i}(A_i)$.
Hence $N_M(A^*) = N_M(Z) = M$ and the first assertion of
Corollary~\ref{corollary3-3} holds by (\ref{control-fusion}).
This also establishes the second assertion unless $p=3$ and
some components $L_i$ are of type $G_2(q)$ or $J_2$, where the
possibility that $\ord{S_i} > 3^3$ in these exceptions
is excluded by the hypotheses of
Corollary~\ref{corollary3-3}.

In the remaining case
let $S^* = S_1 \times \cdots \times S_r$, where $S_1,\dots,S_k$ are the
Sylow 3-subgroups of the components of type $G_2(q)$ or $J_2$, and
$S_{k+1},\dots,S_r$ are the remaining ones.
Again by (\ref{field-action}), $N_M(S) = N_M(S^*)$ so
we must prove the latter normalizer controls strong $M$-fusion in $S$;
indeed, it suffices to prove control of fusion in $S^*$.
Now $N_M(S^*)$ controls strong $M$-fusion in $S^*$ if and only
if the corresponding result holds in each direct factor.
This is trivial for $i > k$ as $S_i$ is normal in that factor.
For $1 \le i \le k$ the result is true since $S_i = 3^{1+2}$,
i.e., $S_i$ has a central series $1 < Z_i < S_i$ whose terms are
all weakly closed in $S_i$ with respect to $N_{L_i}(Z_i)$
(see, for example, \cite{GiSe}).
This establishes the final assertion of
Corollary~\ref{corollary3-3}.
%


\medskip

In Corollary~\ref{corollary3-4} observe that
by Theorem~\ref{theorem3-1},
once $\O_A(G)$ is factored out we have equation
(\ref{corollary3-1-proof-1}) holding, and since $A_F$
acts without fixed points on the cyclic quotient
$D/(D \cap L_1 \cdots L_r)$,
we must have $N_G(A) \le (L_1 \cdots L_r)A_F$.
Thus by (\ref{field-action}) we have
$$
N_G(A) = N_M(A) \le S C_M(A_F) O_{p'}(M).
$$
Since $N_M(A) \cap O_{p'}(M)$ centralizes $A$ we have
$N_G(A) \le S C_M(A_F)$, and hence
$N_G(A) = S C_M(A_F)$.
All parts of Corollary~\ref{corollary3-4} now follow.

\hfill \framebox(7,7){}\break

\section{Examples}
\label{examples}

Throughout this section $p$ is any prime, $G$ is a finite group
possessing a nontrivial Sylow $p$-subgroup $S$.
In this section we describe some families of groups possessing
strongly closed subgroups $A$ contained in $S$.
Let $\omegabar S$
denote the {\it unique smallest strongly closed (with
respect to $G$) subgroup of $S$ that contains $\Omega_1(S)$}.
We focus primarily on groups where
$A = \omegabar S \ne S$, as these groups provide
illuminating examples of fusion, and control (or failure of control)
of fusion in $S$ by $N_G(S)$; and therefore
we describe $N_G(A)$ and $N_G(S)$ in our examples.
In particular, in Section~\ref{exotic} we
show that the extra hypotheses in the last sentence of
Corollary~\ref{corollary3-3} are necessary.
Our constructions are also significant to homotopy theory,
as they provide interesting examples
of cellularizations of classifying spaces,
as detailed in \cite{Flores-Foote08}.

First of all, an example where both $D$ and $A_F$ are nontrivial is
when $G = P\Gamma L_{11}(q)$ with $p=5$ and $q = 3^5$.
Here the simple group
$L = PSL_{11}(q)$ has an abelian Sylow 5-subgroup of type (25,25),
$PGL_{11}(q)/L$ is the cyclic outer diagonal automorphism group of
$L$ of order 11 (this is $DL/L$),
and $\gp f = A_F$ induces a group of order 5 of field automorphisms
on $PGL_{11}(q)$; in particular, $G/L$ is the non-abelian group of order 55.
If $f \in S \in Syl_5(G)$, then $A = \Omega_1(S) = \gp{f,\Omega_1(S \cap L)}$
is elementary abelian of order $5^3$ and strongly closed in $S$ with respect
to $G$, and $A^* = \Omega_1(S \cap L)$ is a minimal strongly closed
subgroup of $G$.

In this example, to compute the normalizers of $A$ and $A^*$
it is easier to work in the
universal group $GL_{11}(q)\gp f$ --- also denoted by $G$ ---
via its action on an 11-dimensional $\FF_q$-vector space $V$
(since the center of $GL_{11}(q)$ has order prime to 5) --- see
the proof of Lemma~\ref{G-not-classical} for some
general methodology.
Let $G^* = GL_{11}(q)$ and $S^* = S \cap G^*$.
Then one sees that $N_G(A^*) = N_G(S^*)$ is contained in a subgroup
$$
H = ((G_1 \times G_2)\gp t \times C)\gp f
$$
where $G_i \iso GL_4(q)$, $C \iso GL_3(q)$,
$t$ interchanges the two factors and $f$ induces field
automorphisms on all three factors and commutes with $t$
(here $G_1 \times G_2 \times C$ acts naturally on a
direct sum decomposition of $V$).
Let $S_i = S \cap G_i$, so $S_i$ is cyclic of order 25
and acts $\FF_q$-irreducibly on the 4-dimensional submodule
for $G_i$.
By basic representation theory,
$C_{G_i}(S_i)$ is cyclic of order $q^4-1$, and
$N_{G_i}(S_i)/C_{G_i}(S_i)$ is cyclic of order 4.
Thus
$$
N_G(A^*) = N_G(S^*) \iso
((q^4-1)\cdot 4 \; \times \; (q^4-1)\cdot 4)\gp{t,f} \times GL_3(3^5).
$$
Since $A_F = \gp f$ acts as a field automorphisms, similar considerations
show that
$$
N_G(A) = S(N_G(S^*) \cap C_{G^*}(f)) \iso
(400 \cdot 4 \; \times \; 400 \cdot 4)\gp{t,f} \times GL_3(3).
$$
The $G$-fusion in $S$ is effected by the group $S(4 \times 4)\gp{t}$,
which is the same for both normalizers.
In this example we may choose $D = [C_G(S^*),f]$, which is of type
$B \times B \times (SL_3(q)\cdot 121)$
where $B$ is cyclic of order $(q^4-1)/5(3^4-1)$;
a (smaller) group of diagonal automorphisms for $D$ could be chosen
inside the abelian factor $B \times B$.

\subsection{Simple groups}
\label{simple-groups}
\    

The following is an
immediate consequence of Theorems~\ref{theorem3-1-p-equal-2}
and \ref{theorem3-1} (here $\O_A(G) = 1$ by
the simplicity of $G$):

\begin{corollary}
\label{corollary5-1}
Let $G$ be a simple group in which $\omegabar S \ne S$.
Then $G$ is isomorphic to one of the groups $L_i$ that appear
in the conclusions of Theorems~\ref{theorem3-1-p-equal-2}
and \ref{theorem3-1}.
In all cases the normalizer of $S$ controls strong fusion in $S$.
\end{corollary}

\begin{proof}
The first assertion is immediate.
Recall that if $G$ is the simple group $G_2(q)$ for some $q$ with
$(q,3) = 1$, then we showed in the proof of
Proposition~\ref{normalizer-of-others-in-theorem3}
(and at the end of the proof of
Lemma~\ref{G-not-classical}) that $S = \Omega_1(S)$.
Thus by Corollary~\ref{corollary3-3}, in all cases in
where $\omegabar S \ne S$
the normalizer of $S$ controls strong fusion in $S$.
\end{proof}

With the exception of the groups of Lie type in characteristic $\ne p$,
the Sylow-$p$ normalizers of the simple groups appearing in the conclusions to
Theorems~\ref{theorem3-1-p-equal-2} and \ref{theorem3-1}
are described explicitly in
Proposition~\ref{normalizer-of-others-in-theorem3}.
We therefore add here only some observations on the structure of the
normalizers in the remaining case.

Let $G$ be a group of Lie type over a
field of characteristic $r \ne p$ and
suppose the Sylow $p$-subgroup $S$ of $G$
is abelian but not elementary abelian (here $p$ is odd).
The overall structure of $N_G(S)$ is governed by the theory of
algebraic groups, as invoked in the proof of
Proposition~\ref{normalizer-of-A-Lie}.
Recapping from that argument:
since the Schur multiplier of $G$ is prime to $p$ we may
work in the universal version of $G$ to describe $N_G(S)$.
Let $\overline G$ be the simply connected universal simple algebraic
group over the algebraic closure of $\FF_r$, and let $\sigma$ be a
Steinberg endomorphism whose fixed points equal $G$.
In the notation of \cite{SS}, $p$ is not a torsion prime for
$\overline G$, so by 5.8 therein $C_{\overline G}(S)$ is a connected,
reductive group whose semisimple component is simply connected.
The general theory of connected, reductive algebraic groups gives
that $C_{\overline G}(S) = \overline Z \, \overline L$, where
$\overline Z$ is the connected component of the center of
$C_{\overline G}(S)$, $\overline L$ is the semisimple component
(possibly trivial), and $\overline Z \cap \overline L$ is a finite group.
Furthermore, $\overline L$ is a product of groups of Lie
type over the algebraic closure of $\FF_r$ of smaller rank than
$\overline G$.  It follows that $C_{\overline G}(S)$ is a commuting
product of the fixed points of $\sigma$ on $\overline Z$ and
$\overline L$, i.e.,
$$
C_{\overline G}(S) = C_{\overline Z}(\sigma) C_{\overline L}(\sigma)
$$
where $S \le C_{\overline Z}(\sigma)$ is an abelian group (a finite torus)
and $C_{\overline L}(\sigma)$ is either solvable or a product of
finite Lie type groups in characteristic $r$.

To complete the generic description of $N_G(S)$
we invoke additional facts from \cite{SS} and \cite[Section 4.10]{GLS3}.
As above, $S$ is contained in a $\sigma$-stable maximal torus
$\overline T_1$, where $\overline T_1$ is obtained from a $\sigma$-stable
split maximal torus $\overline T$ by twisting by some element $w$ of
the Weyl group $W = N_{\overline G}(\overline T)/\overline T$ of
$\overline G$.
Since $S$ is characteristic in the finite torus
$T_1 = \overline{(T_1)}_\sigma$ it follows that $N_G(S)/C_G(S) \iso N_G(T_1)/T_1$.
In most cases, by 1.8 of \cite{SS} or Proposition~3.36 of
\cite{Ca85} we have $N_G(T_1)/T_1 \iso W_\sigma \iso C_W(w)$
(see also \cite[Theorem 2.1.2(d)]{GLS3} and the techniques in the
proof of Theorem~4.10.2 in that volume).

In the special case where $G$ is a classical group (linear, unitary,
symplectic, orthogonal) the normalizer of $S$ can be computed explicitly
by its action on the underlying natural module, $V$, as described in
the proof of Lemma~\ref{G-not-classical}.
In the notation of this lemma, the semisimple component
of order prime to $p$ comes from the normal
subgroup $\text{Isom}(V_0)$ in $\text{Isom}(V)$, where $V_0 = C_V(S)$,
and $S$ is the direct product of the cyclic groups $S \cap \text{Isom}(V_i)$
for $i=1,2,\dots,s$.
The Weyl group normalizing $S$ acts as the symmetric group $S_s$ permuting
the subgroups $\text{Isom}(V_i)$.
The orders of the centralizer and normalizer of a (cyclic) Sylow $p$-subgroup
in each subgroup $\text{Isom}(V_i)$ depend on $p$ and the nature of $G$ ---
Chapter~3 of \cite{Ca85} gives techniques for computing these.

For an easy explicit example of this let $G = SL_{n+1}(q)$ where
$q = r^m$ and $p > n+1$, and assume $p \divides q-1$. In this case
we may choose $S$ contained in the group of diagonal matrices $T$
of determinant 1, which is an abelian group of type
$(q-1,\dots,q-1)$ of rank $n$ (here $T$ is the split torus). In
this case $T = C_G(S)$ and $N_G(S) = N_G(T) = TW$, where $W \iso
S_{n+1}$ is the group of permutation matrices permuting the
entries of matrices in $T$ in the natural fashion (as the ``trace
zero'' submodule of the natural action on the direct product of
$n+1$ copies of the cyclic group of order $q-1$). To obtain the
Sylow $p$-normalizer in the simple group $PSL_{n+1}(q)$ factor out
the subgroup of scalar matrices of order $(n+1,q-1)$.

\subsection{Split extensions}
\    

In this subsection we consider some non-simple groups possessing
strongly closed $p$-subgroups in which $S \ne \omegabar S$.
We show that many split extensions for
which these conditions hold can be constructed. This
construction demonstrates that even when
$N_{\overline G}(\overline S)$ (or $N_{\overline G}(\overline A)$)
controls $\overline G$-fusion in $\overline S$, where overbars denote
passage to $G/O_A(G)$, it need {\it not} be the case that $N_G(A)$
controls fusion in $S$ (or in $A$), even when $N_{\overline
G}(\overline A) = \overline{N_G(A)}$. This highlights the
importance of ``recognizing'' the subgroup $\O_A(G)$ as well as
the isomorphism types of the components of $G/O_A(G)$ in our
classifications.

\begin{proposition}
\label{not-control-fusion-eg}
Let $R$ be any group that is not a $p$-group but is generated
by elements of order $p$.
Assume also that $\omegabar T \ne T$ for some Sylow $p$-subgroup $T$ of $R$.
Let $E$ be any elementary abelian $p$-group on which $R$ acts in such
a way that $R/C_R(E)$ is not a $p$-group.
Let $G$ be the semidirect product $E \sdp R$, and let $S = ET$ be
a Sylow $p$-subgroup of $G$.
Then $G$ is generated by elements of order $p$, $\omegabar S \ne S$,
and $N_G(S)$ does not control fusion in $S$.
\end{proposition}

\begin{proof}
Note that the split extension
$G = ER$ is clearly generated by elements of order $p$
since both $E$ and $R$ are.  Also, $\omegabar S$ contains $E$,
and by Lemma~\ref{basic-facts}, since the extension is split we obtain
$\omegabar S /E \iso \omegabar T < T$, so $\omegabar S \ne S$.
It remains to show that $N_G(S)$ does not control fusion in $S$.

Let $0 = E_0 < E_1< \dots < E_{n-1}< E_n=E$ be a chief series
through $E$, so that each factor $E_i/E_{i-1}$ is an irreducible
$\FF_p R$-module. If each such factor is one-dimensional, then $R$
is represented by upper triangular matrices in its action on $E$.
Since $R$ is generated by elements of order $p$, it must be
represented by unipotent matrices, hence $R/C_R(E)$ is a
$p$-group, a contradiction.

Thus there is some chief factor $E_i/E_{i-1}$ that is not one-dimensional.
If a Sylow normalizer controlled fusion in $S$, then by Lemma~\ref{basic-facts}
the same would be true in the quotient group $G/E_{i-1}$; we show this is
not the case.  To do so, we may pass to the quotient and therefore assume
$E_1$ is a minimal normal, noncentral subgroup of $G$.
Now $Z_1 = Z(S) \cap E_1 \ne 1$ and $Z_1$ is invariant under $N_G(S)$.
However, $R$ acts irreducibly and nontrivially on $E_1$ and $R$ is generated
by conjugates of $S$, so $Z_1 \ne E_1$ and hence $Z_1$ is not $R$-invariant.
Thus for some $z \in Z_1$ and $g \in G$
we must have $z^g \in E_1 - Z_1$, which shows $N_G(S)$ does not
control fusion in $S$.
\end{proof}

This proposition can be invoked to create a host of examples:
Let $R$ be any of the simple groups $L_i$ (or their
quasisimple universal covers) in the conclusion
to Theorem~\ref{theorem3-1} and let $E$ be an $\FF_p R$-module
on which $R$ acts nontrivially (for example, any nontrivial permutation
module).
More specifically, for $p$ odd
let $q$ be any prime power such that $p^2 \divides q-1$,
so that Sylow $p$-subgroups of $R = SL_2(q)$ are cyclic of order $\ge p^2$
(for example, $p=3$ and $q = 19$).
Then $R$ permutes the $q+1$ lines in a 2-dimensional space over $\FF_q$, and
so permutes $q+1$ basis vectors in a $q+1$-dimensional vector
space $E$ over $\FF_p$. Then $G = E\rtimes R$ gives a specific realization
for Proposition~\ref{not-control-fusion-eg}.

Building on the preceding example where $R = SL(2,q)$ for any
prime power $q$ such that $p^2 \divides q-1$: then $T$ may be
represented by diagonal matrices over $\FF_q$, so is cyclic of
order $p^n = \ord{q-1}_p$; moreover, $C_R(T)$ is the group of all
diagonal matrices of determinant 1, hence is cyclic of order
$q-1$. In particular, $\omegabar T = \Omega_1(T) \iso \Z /p$.
Furthermore, $N_R(T) = N_R(\omegabar T)$ is of index 2 in $C_R(T)$
and an involution in $N_R(T)$ inverts $C_R(T)$.
Thus $N_R(\omegabar T)/\omegabar T$ is isomorphic to the dihedral group
of order $2(q-1)/p$.

\subsection{Exotic extensions of $\pmb{G_2(q)}$}
\label{exotic}
\    

When $G$ is the simple group $G_2(q)$ for some $q$ with $(q,3) = 1$,
although a Sylow 3-subgroup $S$ contains a strongly closed
subgroup $A = Z(S)$ of order $p = 3$, when we impose the
additional hypothesis that our strongly closed subgroup must
contain all elements of order 3 the strongly closed subgroup $A$
does not arise in our considerations because $S = \Omega_1(S)$.
For the same reason, if $G = ER$ is any split extension of
$R = G_2(q)$ by an elementary abelian 3-group
and $S = ET$ for $T \in Syl_3(R)$, then again
$S = \Omega_1(S) = \omegabar S$. In this subsection we describe a
family of extensions that we call ``half-split" in the sense that
they split over a certain conjugacy class of elements of $R$ but
do not split over another.  In this way we construct extensions
$G$ of $R = G_2(q)$ by certain elementary abelian 3-groups $E$
such that for $S \in Syl_3(G)$ we have $\Omega_1(S)/E$ mapping
onto the strongly closed subgroup of order 3 in a Sylow
3-subgroup $S/E$ of $G_2(q)$.  In particular, these ``exotic''
extensions show that the exceptional case of
Corollary~\ref{corollary3-3} cannot be removed:
when $9 \divides q^2 - 1$ these groups $G$ are generated by
elements of order 3, have $\omegabar S \ne S$,
but $N_{G/E}(S/E)$ does not control fusion in
$S/E$ (here $E = \O_A(G)$ where $A = \omegabar S$).

The following general proposition will construct such extensions.

\begin{proposition}
\label{nonsplit-construction}
Let $p$ be a prime dividing the order of the finite group $R$
and let $X$ be a subgroup of order $p$ in $R$.
Then there is an $\FF_p R$-module $E$ and an extension
$$
1 \longrightarrow E \longrightarrow  G \longrightarrow  R \longrightarrow  1
$$
of $R$ by $E$ such that the
extension of $X$ by $E$ does not split, but the extension of $Z$ by
$E$ splits for every subgroup $Z$ of order $p$ in $R$ that is not conjugate
to $X$.
In particular, for nonidentity elements $x \in X$ and $z \in Z$ every element
in the coset $xE$ has order $p^2$ whereas $zE$ contains elements of order $p$
in $G$.
\end{proposition}

\begin{proof}
Let $E_0$ be the one-dimensional trivial $\FF_p X$-module.
By the familiar cohomology of cyclic groups (\cite{Br82}, Section III.1):
\begin{equation}
H^2(X,E_0) \iso \ZZ/p\ZZ
\label{cohom-1}
\end{equation}
and a non-split extension of $X$ by $E_0$ is just a cyclic group of order $p^2$.
Now let
$$
E = \Coind_X^R E_0 = \Hom_{\ZZ X}(\ZZ R,E_0)
$$
be the coinduced module from $X$ to $R$ (which is isomorphic to the
induced module $E_0 \tensor_{\FF_p X} \FF_p R$ in the case of
finite groups), so that $E$ has $\FF_p$-dimension $\frac 1p \ord R$.
By Shapiro's Lemma (\cite{Br82}, Proposition~III.6.2)
\begin{equation}
H^2(R,E) \iso H^2(X,E_0).
\label{cohom-2}
\end{equation}
Thus by (\ref{cohom-1}) there is a non-split extension of $R$ by $E$ ---
call this extension group $G$ and identify $E$ as a normal subgroup of
$G$ with quotient group $G/E = R$.

The isomorphism in Shapiro's Lemma, (\ref{cohom-2}), is given by
the compatible homomorphisms
$\iota : X \hookrightarrow R$ and $\pi : \Coind_X^R E_0 \rightarrow E_0$,
where $\pi$ is the natural map $\pi(f) = f(1)$.
In particular, this isomorphism is a composition
$$
H^2(R,E) \overset {\text{res}}\longrightarrow H^2(X,E)
\overset {\pi^*} \longrightarrow H^2(X,E_0).
$$
Thus the 2-cocycle defining the non-split extension group $G$,
which maps to a nontrivial element in $H^2(X,E_0)$,
by restriction gives a non-split extension of $X$ by $E$ as well.

For any subgroup $Z$ of $R$ of order $p$
with $Z$ not conjugate to $X$,
by the Mackey decomposition for induced representations
\begin{equation}
\Res_Z^R \, \Ind_X^R \, E_0 = \bigoplus_{g \in \mathcal R}
\Ind_{Z \cap g X g^{-1}}^Z \, \Res_{Z \cap g X g^{-1}}^{g X g^{-1}} g E_0
\label{cohom-3}
\end{equation}
where $\mathcal R$ is a set of representatives for the $(Z,X)$-double cosets in $R$.
By hypothesis, $Z \cap gXg^{-1} = 1$ for every $g \in R$,
hence each term in the direct sum on the right hand side
is an $\FF_p Z$-module obtained by inducing a one-dimensional trivial
$\FF_p$-module for the identity subgroup to a $p$-dimensional $\FF_p Z$-module,
i.e., is a free $\FF_p Z$-module of rank 1.
(Alternatively, $E$ is the $\FF_p$-permutation module for the action of $R$ by
left multiplication on the left cosets of $X$;
by the fusion hypothesis, $Z$ acts on a basis of $E$ as a product of
disjoint $p$-cycles with no 1-cycles.)
This shows $E$ is a free $\FF_p Z$-module, and hence the extension
of $Z$ by $E$ splits.
This completes the proof.
\end{proof}

The $p^{\text{th}}$-power map on elements in the lift of $X$ to
$G$ can be described more precisely.
By the Mackey decomposition in (\ref{cohom-3}) inducing from $X$
but rather restricting to $X$ instead of $Z$, or by direct inspection of
the action of $X$ on the $\FF_p$-permutation module $E$, we see that
$E$ decomposes as an $\FF_p X$-module direct sum as
$$
E = E_1 \oplus E_2,
$$
where $E_1$ is a trivial $\FF_p X$-module and $E_2$ is a free $\FF_p X$-module.
Since $X$ splits over the free summand $E_2$,
we see that $X$ does not split over $E_1$, and hence
$$
XE_1 \iso (\Z /{p^2}) \times \Z /p \times \dots \times \Z /p
\qquad\text{with}\qquad E_1 = \Omega_1(XE_1).
$$
Thus for every element $x$ in $G - E$ mapping to an element of $X$ in $G/E$,
$x^p$ has a nontrivial component in $E_1$.

\medskip

One may also observe that by
taking direct sums we can arrange more generally that if
$X_1,X_2,\dots,X_n$ are representatives of the distinct conjugacy classes of
subgroups of order $p$ in $R$, then for any $i \in \{1,2,\dots,n\}$
there is an $\FF_p R$-module $E$ and an
extension of $R$ by $E$ such that in the extension group
each of $X_1,\dots,X_i$ splits over $E$ but none
of $X_{i+1},\dots,X_n$ do.

\medskip

We are particularly interested in the case $R = G_2(q)$
with $p = 3$ and $(q,3) = 1$.
The normalizer of a Sylow 3-subgroup of $R$ is described in
Proposition~\ref{normalizer-of-others-in-theorem3}:
Let $T \in Syl_3(R)$ and let $Z = Z(T) = \gp z$.
In the notation preceding
Proposition~\ref{irreducible-action},
$N_R(Z) \iso SL_3^\epsilon(q)\cdot 2$ according as
$3 \divides q-\epsilon$.
Moreover, if $9 \divides q-\epsilon$ then
$N_R(T)$ does not control fusion in $T$:
all elements of order 3 in $T-Z$ are conjugate in
$C_R(Z)$ whereas by
Proposition~\ref{normalizer-of-others-in-theorem3},
$N_R(T)/T$ has order 4 for this congruence of $q$.

Now consider the extension group $G$ constructed in
Proposition~\ref{nonsplit-construction} with $p=3$, $R = G_2(q)$,
$Z = \gp z$ and $X = \gp x$ for any $x \in T - Z$ of order 3.
Let $S \in Syl_3(G)$ with $S$ mapping onto $T$ in $G/E \iso R$.
Since Proposition~\ref{normalizer-of-others-in-theorem3}
shows all elements of order 3 in
$T - Z$ are conjugate to $x$ but not to $z$,
the structure of the extension implies that
$A = \Omega_1(S) = \omegabar S$ contains $E$ and
maps to $Z$ in $S/E$.
Thus $\O_A(G) = E$ and $\overline A = \overline Z$.
By Corollary~\ref{corollary3-3}, the normalizer of $Z$ in
$R = G_2(q)$ controls 3-fusion in $G_2(q)$, so in particular
$SL^*_3(q)$ has the same mod 3 cohomology as $G_2(q)$, where
$SL^*_3(q)$ denotes the group $SL_3^\epsilon(q)$ together
with the outer (graph) automorphism of order 2 inverting its
center ($N_R(Z) \iso SL^*_3(q)$). On the other hand, $Z$ is normal
in $SL^*_3(q)$, and $SL^*_3(q)/Z$ is isomorphic to $PSL^*_3(q)$.

This example highlights the importance of having a
classification of \emph{all} groups possessing a nontrivial
strongly closed $p$-subgroup that is not Sylow ---
not just the simple groups having
such a subgroup that contains $\Omega_1(S)$ --- since
the subgroup $\omegabar S$ does not pass in a transparent
fashion to quotients.

\medskip

The extensions of our techniques and results to more general
$p$-local spaces with a notion of $p$-fusion seem to be the
natural next step of our study; in particular, classifying spaces
of $p$-local finite groups and some families of non-finite groups
offer enticing possibilities.


\bigskip

\bigskip\noindent
Ram\'on J.~Flores\\
Departamento de Estad\'\i stica, Universidad Carlos III
de Madrid, C/ Madrid 126\\
E -- 28903 Colmenarejo (Madrid) --- Spain \\
e-mail: {\tt rflores@est-econ.uc3m.es}

\medskip

\noindent
Richard M.~Foote\\
Department of Mathematics and Statistics, University of Vermont, \\
16 Colchester Avenue,
Burlington, Vermont 05405 --- U.S.A. \\
e-mail: {\tt foote@math.uvm.edu}




\begin{thebibliography}{draft7}

\bibitem[Asc94]{AS}
M.~Aschbacher.
\newblock {\it Sporadic Groups}.
\newblock Cambridge University Press, Cambridge Tracts in Mathematics Series, \#104, 1994.

\bibitem[Asc07]{AS07}
M.~Aschbacher.
\newblock Notes on the local theory of saturated fusion systems.
\newblock Unpublished lecture notes for the Univ. of Birmingham short course
on fusion systems, 31 July---4 August, 2007.

\bibitem[BCGLO07]{Broto207}
C. Broto, N. Castellana, J. Grodal, R. Levi and B. Oliver.
\newblock Extensions of $p$--local finite groups.
\newblock {\em Trans. Amer. Math. Soc.} 359:3791--3858, 2007.


\bibitem[BLO03]{Broto032}
C.~Broto, R.~Levi, and B.~Oliver.
\newblock The homotopy theory of fusion systems.
\newblock {\em J. Amer. Math. Soc.}, 16(4):779--856 (electronic), 2003.

\bibitem[BLO07]{Broto07}
C.~Broto, R.~Levi, and B.~Oliver.
\newblock Discrete models for
the $p$--local homotopy theory of compact Lie groups and
$p$--compact groups.
\newblock {\em Geometry and Topology}, 11:315--427, 2007.


\bibitem[Bro82]{Br82}
K.~S.~Brown.
\newblock {\it Cohomology of Groups}.
\newblock Springer-Verlag Graduate Texts in Math. \#87, New York, 1982.

\bibitem[Ca85]{Ca85}
R.~W.~Carter.
\newblock {\it Finite Groups of Lie Type}.
\newblock Wiley Interscience, New York, 1985.

\bibitem[CCNPW]{Atlas}
J.~Conway, R.~Curtis, S.~Norton, R.~Parker and R.~Wilson.
\newblock {\it Atlas of Finite Groups}.
\newblock Clarendon Press, Oxford, 1985.


\bibitem[CKS76]{CKS}
C.~Curtis, W.~Kantor and G.~Seitz.
\newblock The 2-transitive permutation representations of the finite Chevalley groups.
\newblock Trans. Amer. Math. Soc. 218(1976), 1--59.

\bibitem[Flo07]{Ramon}
R.J.~Flores.
\newblock Nullification and cellularization of classifying spaces of finite
  groups.
\newblock {\em Trans. Amer. Math. Soc.}, 359, 1791--1816, 2007.

\bibitem[FS07]{Flores07}
R.J.~Flores and J.~Scherer.
\newblock Cellularization of classifying spaces and fusion properties of finite groups.
\newblock To appear in {\em J. London Math. Soc.}. Available at
ArXiv, http://arxiv.org/pdf/math.AT/0501442.

\bibitem[FlFo08]{Flores-Foote08}
R.J.~Flores and R.~Foote.
\newblock Cellularization of classifying spaces of finite groups.
\newblock To appear, 2008.

\bibitem[Foo97]{Foote97}
R.~Foote.
\newblock A characterization of finite groups containing a strongly closed {$2$}-subgroup.
\newblock {\em Comm. Algebra}, 25(2):593--606, 1997.

\bibitem[Foo97a]{Foote97a}
R.~Foote.
\newblock Sylow 2-subgroups of Galois Groups Arising as Minimal Counterexamples
to Artin's Conjecture.
\newblock {\em Comm. Algebra}, 25(2):607--616, 1997.

\bibitem[GiSe85]{GiSe}
A.~Gilotti and L.~Serena.
\newblock Strongly closed subgroups and strong fusion.
\newblock J. London Math. Soc. 32(1985), 103--106.


\bibitem[Gla66]{Gla66}
G.~Glauberman
\newblock Central elements in core-free groups.
\newblock J. Algebra 4(1966), 403--420.

\bibitem[Gol74]{Gol74}
D.~Goldschmidt
\newblock 2-Fusion in finite groups.
\newblock Ann. Math. 99(1974), 70--117.

\bibitem[Gol75]{Gol75}
D.~Goldschmidt
\newblock Strongly closed 2-subgroups of finite groups.
\newblock Ann. Math. 102(1975), 475--489.


\bibitem[GL83]{GL}
D.~Gorenstein and R.~Lyons.
\newblock The local structure of finite groups of characteristic 2 type.
\newblock Mem. Amer. Math. Soc. 276(1983), 1--731.

\bibitem[GLS93]{GLS-BN}
D.~Gorenstein, R.~Lyons and R.~Solomon.
\newblock On split $(B,N)$-pairs of rank 2.
\newblock J. Algebra 157(1993), 224--270.

\bibitem[GLSv3]{GLS3}
D.~Gorenstein, R.~Lyons and R.~Solomon.
\newblock {\it The classification of the finite simple groups}.
\newblock AMS Surveys \&\ Monographs, Vol. 40, No. 3(1998), 1--419.

\bibitem[GLSv4]{GLS4}
D.~Gorenstein, R.~Lyons and R.~Solomon.
\newblock {\it The classification of the finite simple groups}.
\newblock AMS Surveys \&\ Monographs, Vol. 40, No. 4(1999), 1--341.

\bibitem[Gor80]{Gor80}
D.~Gorenstein.
\newblock {\em Finite groups}.
\newblock Chelsea Publishing Co., New York, second edition, 1980.

\bibitem[HKS72]{HKS72}
C.~Herring, W.~Kantor and G.~Seitz.
\newblock Finite groups with a split BN-pair of rank 1.
\newblock J. Algebra, 20(1972), 435--475.

\bibitem[KL90]{KL}
P.~Kleidman and M.~Liebeck.
\newblock {\it The Subgroup Structure of the Finite Classical Groups}.
\newblock LMS Lecture Notes No. 129, CU Press, 1990.


\bibitem[Lin06]{Linckelmann06}
M. Linckelmann.
\newblock Simple fusion systems and the Solomon 2-local groups.
\newblock {\em J. of Algebra}, 296(2):385-401, 2006.

\bibitem[SS70]{SS}
T.~Springer and R.~Steinberg,
\newblock Conjugacy classes.
\newblock in {\it Seminar on Algebraic Groups and Related Finite Groups}
ed. by A. Borel et. al., Springer Lecture Notes \#131,
Springer-Verlag, Berlin, 1970.

\bibitem[Suz62]{Suz62}
M.~Suzuki,
\newblock On a class of doubly transitive groups.
\newblock Ann. of Math., 75(1962), 105--145.

\bibitem[Wa66]{Wa66}
H.~N.~Ward,
\newblock On Ree's series of simple groups.
\newblock Trans. Amer. Math. Soc., 121(1966), 62--89.

\bibitem[Wi98]{W}
R.~A.~Wilson,
\newblock Some subgroups of the Thompson group.
\newblock J. Australian Math. Soc., 44(1998), 17--32.

\end{thebibliography}
\end{document}